\documentclass[10pt]{amsart}
\usepackage{amscd,amssymb,hyperref}
\usepackage[PostScript=dvips]{diagrams}

\usepackage[dvips]{graphicx}

\usepackage{graphics}
\usepackage{epsfig}
\usepackage{picinpar}
\usepackage{wrapfig}
\usepackage{amsmath}
\usepackage{color}
\usepackage{pstricks,pst-node,pst-text,pst-3d}
\usepackage{pst-plot}
\usepackage{verbatim}

\newtheorem{thm}{Theorem}[section]
\newtheorem{lem}[thm]{Lemma}
\newtheorem{cor}[thm]{Corollary}

\newtheorem{hypothesis}{Hypothesis}[section]
\def\square{\vbox{
      \hrule height 0.4pt
      \hbox{\vrule width 0.4pt height 5.5pt \kern 5.5pt \vrule width 0.4pt}
      \hrule height 0.4pt}}

\def\id{\mathrm{id}}

\def\Ker{\mathrm{K er}}
\def\ch\mathrm{c h}

\newcommand{\Brun}{\mathrm{Brun}}

\newcommand{\R}{\ensuremath{\mathbb{R}}}

\newcommand{\RP}{\mathbb{R}\mathrm{P}}

\let\la=\langle
\let\ra=\rangle

\numberwithin{equation}{section}

\begin{document}

\newcommand{\auths}[1]{\textrm{#1},}
\newcommand{\artTitle}[1]{\textsl{#1},}
\newcommand{\jTitle}[1]{\textrm{#1}}
\newcommand{\Vol}[1]{\textbf{#1}}
\newcommand{\Year}[1]{\textrm{(#1)}}
\newcommand{\Pages}[1]{\textrm{#1}}

\author{J. Y. Li$^*$}
\address{ Institute of Mathematics and Physics, Shijiazhuang Railway Institute, Shijiazhuang 050043, China}\email{yanjinglee@163.com}
\author{J. Wu$^{\dag}$}
\address{Department of Mathematics\\
National University of Singapore\\
Block S17 (SOC1)\\
10, Lower Kent Ridge Road\\
Republic of Singapore}\email{matwuj@nus.edu.sg}
\urladdr{http://www.math.nus.edu.sg/\~{}matwujie}

\thanks{$^*$Partially supported by the National Natural Science Foundation of China 10971050}
\thanks{$^\dag$Partially supported by the Academic Research Fund
of the National University of Singapore R-146-000-101-112}

\subjclass[2000]{Primary 55Q40,20F12; Secondary 20F36, 57M25}

\keywords{symmetric commutator subgroup, homotopy group, link group, Brunnian braid, free group, surface group}

\begin{abstract}
In this paper, we investigate some applications of commutator
subgroups to homotopy groups and geometric groups. In particular, we
show that the intersection subgroups of some canonical subgroups in
certain link groups modulo their symmetric commutator subgroups are
isomorphic to the (higher) homotopy groups. This gives a connection
between links and homotopy groups. Similar results hold for braid
and surface groups.
\end{abstract}

\title[On symmetric commutator subgroups and homotopy groups]{On symmetric commutator subgroups, braids, links and homotopy groups}

\maketitle

\section{Introduction}

The purpose of this article is to investigate some applications of
commutator subgroups to homotopy groups and geometric groups.

Recall ~\cite[p. 288-289]{MKS} that \textit{a bracket arrangement of weight $n$} in a group $G$ is a map
$\beta^n\colon G^n\to G$ which is defined inductively as follows:
$$\beta^1=\id_G,\,\, \beta^2(a_1,a_2)=[a_1,a_2]$$
for any $a_1,a_2\in G$, where $[a_1,a_2]=a_1^{-1}a_2^{-1}a_1a_2$. Suppose that the bracket arrangements of weight $k$ are
defined for $1\leq k<n$ with $n\geq 3$. A map $\beta^n\colon G^n\to G$ is called
a bracket arrangement of weight $n$ if $\beta^n$ is the composite
\begin{diagram}
G^n=G^k\times G^{n-k}&\rTo^{\beta^k\times\beta^{n-k}}&G\times G&\rTo^{\beta^2}&G
\end{diagram}
for some bracket arrangements $\beta^k$ and $\beta^{n-k}$ of weight $k$ and
$n-k$, respectively, with $1\leq k<n$. For instance, if $n=3$, there are two bracket arrangements given by
$[[a_1,a_2],a_3]$ and $[a_1,[a_2,a_3]]$.

Let $R_{j}$ be a sequence of subgroups of $G$ for
 $1\leq j\leq n$. The \textit{fat commutator subgroup}  $[[R_{1}, R_2, \dots,R_{n}]]$ is defined to be the subgroup of $G$ generated by all of
the commutators
\begin{equation}\label{equation1.1}
\beta^t(g_{i_1},\ldots,g_{i_t}),
\end{equation}
where
\begin{enumerate}
\item[1)] $1\leq i_s\leq n$;
\item[2)] $\{i_1,\ldots,i_t\}=\{1,\ldots,n\}$, that is each integer in $\{1,2,\cdots,n\}$ appears as at least one of the
integers $i_s$;
\item[3)] $g_j\in R_j$;
\item[4)] $\beta^t$ runs over all of the  bracket arrangements of weight $t$ (with $t\geq n$).
\end{enumerate}
For convenience, let $[[R_1]]=R_1$. The fat commutator subgroups have important applications in homotopy theory~\cite{BCWW,BL,EM,Wu1}. Let $G=F_n$ be the free group of rank $n$ with a basis $x_1,\ldots,x_n$. Let $R_i=\la\la x_i\ra\ra$ be the normal closure of $x_i$ in $F_n$ for $1\leq i\leq n$ and let $R_{n+1}=\la\la x_1\cdots x_n\ra\ra$ be the normal closure of the product element $x_1x_2\cdots x_n$ in $F_n$. According to~\cite[Theorem 1.7]{Wu1}, the homotopy group $\pi_{n+1}(S^2)$ is isomorphic to the quotient group $$ (R_1\cap R_2\cap\cdots\cap R_{n+1})/[[R_1,R_2,\ldots,R_{n+1}]]$$ for each $n$.

To understand the fat commutator subgroup, we consider the
\textit{symmetric commutator subgroup}\footnote{Symmetric commutator
subgroup was named by Roman Mikhailov during the communications.}
$[[R_1, R_2],\ldots,R_n]_S$ defined by
$$
[[R_1,R_2],\ldots,R_n]_S=\prod_{\sigma\in \Sigma_n}[[R_{\sigma(1)},R_{\sigma(2)}],\ldots,R_{\sigma(n)}],
$$
where $[[R_{\sigma(1)},R_{\sigma(2)}],\ldots, R_{\sigma(n)}]$ is the subgroup generated by the left iterated commutators
$$
[[[g_1,g_2],g_3],\ldots,g_n]
$$
with $g_i\in R_{\sigma(i)}$. For convenience, let $[R_1]_S=R_1$. From the definition, the symmetric commutator subgroup is a subgroup of the fat commutator subgroup. Our first result states that the fat commutator subgroup is in fact the same as the symmetric commutator subgroup.

\begin{thm}\label{theorem1.1}
Let $R_j$ be any normal subgroup of a group $G$ with $1\leq j\leq n$. Then
$$
[[R_1,R_2,\ldots,R_n]]=[[R_1,R_2],\ldots,R_n]_S.
$$
\end{thm}

The symmetric commutator subgroup can be simplified as follows:

\begin{thm}\label{theorem1.2}
Let $R_j$ be any normal subgroup of a group $G$ with $1\leq j\leq n$. Then
$$
[[R_1,R_2],\ldots,R_n]_S=\prod_{\sigma\in \Sigma_{n-1}}[[R_{1},R_{\sigma(2)}],\ldots,R_{\sigma(n)}],
$$
where $\Sigma_{n-1}$ acts on $\{2,3,\ldots,n\}$.
\end{thm}

Our next step is to give a generalization of~\cite[Theorem 1.7]{Wu1}. This will give more connections between homotopy groups and symmetric commutator subgroups.

Let $(X,A)$ be a pair of spaces. An  \textit{$n$-partition} of $X$ relative to $A$ means a sequence of subspaces $(A_1,\ldots,A_n)$ of $X$ such that
\begin{enumerate}
\item $A=A_i\cap A_j$ for each $1\leq i<j\leq n$ and
\item $X=\bigcup\limits_{i=1}^n A_i$.
\end{enumerate}
An $n$-partition $(A_1,\ldots,A_n)$ of $X$ relative to $A$ is called \textit{cofibrant} if the inclusion
$$
\bigcup_{i\in I}A_i \hookrightarrow \bigcup_{j\in J}A_j
$$
is a cofibration for any $I\subseteq J\subseteq\{1,2,\ldots,n\}$. Note that for a cofibrant partition, each union $A_I=\bigcup_{i\in I}A_i$ is the homotopy colimit of the diagram given by the inclusions $A_{I'}\hookrightarrow A_{I''}$ for $\emptyset\subseteq I'\subseteq I''\subsetneqq I$.

\begin{thm}\label{theorem1.3}
Let $(X,A)$ be a pair of spaces and let $(A_1,\ldots,A_n)$ be a cofibrant $n$-partition of $X$ relative to $A$ with $n\geq 2$. Suppose that
\begin{enumerate}
\item[i)] For any proper subset $I=\{i_1,\ldots,i_k\}\subsetneq \{1,2,\ldots,n\}$, the union
$
\bigcup_{i\in I}A_i
$
is a path-connected $K(\pi,1)$-space.
\item[ii)] The inclusion $A \to A_i$ induces an epimorphism of the corresponding fundamental groups for each $1\leq i\leq n$.
\end{enumerate}
Let $R_i$ be the kernel of $\pi_1(A)\to \pi_1(A_i)$ for $1\leq i\leq n$. Then
\begin{enumerate}
\item For any proper subset $I=\{i_1,\ldots,i_k\}\subsetneq \{1,2,\ldots,n\}$,
$$
R_{i_1}\cap\cdots\cap R_{i_k}=[[R_{i_1}, R_{i_2}],\ldots,R_{i_k}]_S.
$$
\item For any $1<k\leq n$ and any subset $I=\{i_1,\ldots,i_k\}\subseteq \{1,2,\ldots,n\}$, there is an isomorphism of groups
$$
\pi_k(X)\cong \left.\left(\bigcap_{s=1}^k\left(R_{i_s}\cdot \prod_{j\in J}R_j\right)\right)\right/\left([[R_{i_1}, R_{i_2}],\ldots,R_{i_k}]_S\cdot \prod_{j\in J}R_j \right),
$$
where $J= \{1,2,\ldots,n\}-I$.
In particular, $$\pi_n(X)\cong (R_1\cap R_2\cap \cdots\cap R_n)/[[R_1,R_2],\ldots,R_n]_S.$$
\end{enumerate}
\end{thm}

 For the special case where $n=2$, Theorem~\ref{theorem1.3} is the classical Brown-Loday Theorem~\cite[Corollary 3.4]{BL} with a generalization recently given in~\cite{EM}. But the connectivity hypothesis on the subgroups in~\cite{EM} seems difficult to check. Theorem~\ref{theorem1.3} emphasizes the point that, for any space $X$ that admits a cofibrant $K(\pi,1)$-partition, the higher homotopy groups of $X$ measure the difference between the intersection subgroups and the symmetric commutator subgroups for certain subgroups in the fundamental groups of the partition spaces. Moreover Theorem~\ref{theorem1.3} admits many applications. A direct consequence is to give an interesting connection between link groups and higher homotopy groups.

\begin{cor}\label{corollary1.4}
Let $M$ be a path-connected $3$-manifold and $L$ be a proper $m$-link in $M$ with $m\geq 2$. Suppose that for any nonempty sub-link $L'$ of $L$, the link complement $M\smallsetminus |L'|$ is a $K(\pi,1)$-space. Let $\Lambda_1,\Lambda_2,\ldots,\Lambda_n$ be any subsets of $\{1,2,\ldots,m\}$, with $n\geq 2$, such that
\begin{enumerate}
\item[(i)] $\Lambda_i\not=\emptyset$ for each $1\leq i\leq n$.
\item[(ii)] $\Lambda_i\cap \Lambda_j=\emptyset$ for $i\not=j$.
\item[(iii)] $\bigcup_{i=1}^n \Lambda_i=\{1,2,\ldots,m\}$.
\end{enumerate}
Let $\alpha_j$ be the $j\,$th meridian of the link $L$ and let
$$
R_i=\la\la \alpha_j \ | \ j\in \Lambda_i\ra\ra
$$
be the normal closure of $\alpha_j$ with $j\in \Lambda_i$ in $\pi_1(M\smallsetminus |L|)$. Then
\begin{enumerate}
\item For any proper subset $I=\{i_1,\ldots,i_k\}\subsetneq \{1,2,\ldots,n\}$,
$$
R_{i_1}\cap\cdots\cap R_{i_k}=[[R_{i_1}, R_{i_2}],\ldots,R_{i_k}]_S.
$$
\item There is an isomorphism of groups
$$
(R_1\cap R_2\cap\cdots\cap R_n)/[[R_1,R_2],\ldots,R_n]_S\cong \pi_n(M).
$$\hfill $\Box$
\end{enumerate}
\end{cor}

There are examples of links whose complements are $K(\pi,1)$-space
s. For instance, let $M=S^3$ and let $\pi\colon S^3\to S^2$ be the
Hopf fibration. Let $q_1,\ldots,q_m$ be $m$ distinct points in
$S^2$. Let $L=\pi^{-1}(\{q_1,q_2,\ldots,q_m\})$. Then $L$ is an
$m$-link in $S^3$ with the property that $S^3\smallsetminus |L'|$ is
a $K(\pi,1)$-space  for any non-empty sub-link $L'$ of $L$.

Consider the special case where $n=m$ with $\Lambda_i=\{i\}$ in Corollary~\ref{corollary1.4}. Then $R_i$ is the normal closure generated by the $i\,$th meridian in $\pi_1(M\smallsetminus |L|)$. Any element in the intersection subgroup $R_1\cap R_2\cap\cdots\cap R_n$ can be represented by a $1$-link $l$. The union $L\cup \{l\}$ gives an $(n+1)$-link in $M$ related to Brunnian links. Thus Corollary~\ref{corollary1.4} gives a connection between higher homotopy groups and Brunnian links.

For more applications of Theorem~\ref{theorem1.3}, we consider certain subgroups of surface groups and braid groups whose intersection subgroups modulo symmetric commutator subgroups are given by the homotopy groups. Our results on the braid groups are described as follows.

Let $M$ be a manifold. Recall that the $m\,$th ordered configuration space $F(M,m)$ is defined by
$$
F(M,m)=\{(z_1,\ldots,z_m)\in M^m\ | \ z_i\not=z_j\textrm{ for } i\not=j\}
$$
with subspace topology. Recall that the Artin braid group $B_n$ is generated by $\sigma_1,\ldots,\sigma_{n-1}$ with defining relations given by $\sigma_i\sigma_j=\sigma_j\sigma_i$ for $|i-j|\geq 2$ and $\sigma_i\sigma_{i+1}\sigma_i=\sigma_{i+1}\sigma_i\sigma_{i+1}$ for each $i$. The pure braid group $P_n$ is defined to be the kernel of the canonical quotient homomorphism from $B_n$ to the symmetric group $\Sigma_n$, with a set of generators given by
$$
A_{i,j}=\sigma_{j-1}\sigma_{j-2}\cdots \sigma_{i+1}\sigma_i^2\sigma_{i+1}^{-1}\cdots\sigma_{j-2}^{-1}\sigma_{j-1}^{-1}
$$
for $1\leq i<j\leq n$. Let
\begin{align*}
A_{0,j}  &  =(A_{j,j+1}A_{j,j+2}\cdots A_{j,n})^{-1}(A_{1,j}\cdots
A_{j-1,j})^{-1}\\
&  =(\sigma_{j}\sigma_{j+1}\cdots\sigma_{n-2}\sigma_{n-1}^{2}\sigma
_{n-2}\cdots\sigma_{j})^{-1}\cdot(\sigma_{j-1}\cdots\sigma_{2}\sigma_{1}
^{2}\sigma_{1}\cdots\sigma_{j-1})^{-1}.
\end{align*}
Then we have the following:
\begin{thm}\label{theorem1.5}
Let $\Lambda_1,\Lambda_2,\ldots,\Lambda_n$ be any subsets of $\{1,2,\ldots,m\}$, with $n\geq 2$, such that
\begin{enumerate}
\item[(i)] $\Lambda_i\not=\emptyset$ for each $1\leq i\leq n$.
\item[(ii)] $\Lambda_i\cap \Lambda_j=\emptyset$ for $i\not=j$.
\item[(iii)] $\bigcup_{i=1}^n \Lambda_i=\{1,2,\ldots,m\}$.
\end{enumerate}
Let $R_i=\la\la A_{0,j} \ | \ j\in \Lambda_i\ra\ra$ be the normal closure in the pure braid group $P_m$. Then
\begin{enumerate}
\item For any proper subset $I=\{i_1,\ldots,i_k\}\subsetneq \{1,2,\ldots,n\}$,
$$
R_{i_1}\cap\cdots\cap R_{i_k}=[[R_{i_1}, R_{i_2}],\ldots,R_{i_k}]_S.
$$
\item There is an isomorphism of groups
$$
(R_1\cap R_2\cap\cdots\cap R_n)/[[R_1,R_2],\ldots,R_n]_S\cong \pi_n(F(S^2,m)).
$$
\end{enumerate}
\end{thm}

For instance, for $n=2$ with $m\geq 3$, $(R_1\cap R_2)/[R_1,R_2]=\pi_2(F(S^2,m))=0$, where $\pi_2(F(S^2,m))=0$ for $m\geq 3$ is given in~\cite[p. 244]{FB}. For $m\geq n\geq 3$, from~\cite[Theorem 1]{Fadell}, we have $\pi_n(F(S^2,m))\cong \pi_n(S^2)$ and so
$$
(R_1\cap R_2\cap\cdots\cap R_n)/[[R_1,R_2],\ldots,R_n]_S\cong \pi_n(F(S^2,m))\cong \pi_n(S^2).
$$
The braided descriptions of the homotopy groups $\pi_*(S^2)$ have
been studied in~\cite{BCWW,BMVW, CW1,CW2,LW, Wu2}.
Theorem~\ref{theorem1.5} is a new braided description of the
homotopy groups.

The article is organized as follows. The proofs of Theorems~\ref{theorem1.1} and~\ref{theorem1.2} are given in section~\ref{section2}. In section~\ref{section3}, we provide the proof of Theorem~\ref{theorem1.3}. In section~\ref{section4}, we discuss some applications to the free groups and surface groups. The applications to the braid groups and the proof of Theorem~\ref{theorem1.5} are given in section~\ref{section5}.

\section{The Proofs of Theorems~\ref{theorem1.1} and~\ref{theorem1.2}}\label{section2}
\subsection{Some Lemmas}
In this subsection, we give some useful lemmas on commutator subgroups. The following lemma is elementary.

\begin{lem}\label{lemma2.1}
Let $G$ be a group and let $A,B,C$ be normal subgroups of $G$. Then $[AB,C]=[A,C][B,C]$.\hfill $\Box$
\end{lem}

The following classical theorem can be found in~\cite[Theorem 5.2, p.290]{MKS}.

\begin{thm}[Hall's Theorem]
Let $G$ be a group and let $A,B,C$ be normal subgroups of $G$. Then any one of the subgroups $[A,[B,C]]$, $[[A,B],C]$ and $[[A,C],B]$ is a subgroup of the product of the other two.
\hfill $\Box$
\end{thm}

Now let $R_1,\ldots,R_n$ be subgroups of $G$. Recall that the symmetric commutator subgroup $[[R_1,R_2],\ldots,R_n]_S$ is defined by
$$
[[R_1,R_2],\ldots,R_n]_S=\prod_{\sigma\in \Sigma_n}[[R_{\sigma(1)},R_{\sigma(2)}],\ldots,R_{\sigma(n)}].
$$
If each $R_i$ is a normal subgroup of $G$, then $[[R_{\sigma(1)},R_{\sigma(2)}],\ldots,R_{\sigma(n)}]$ is normal in $G$ and so is $[[R_1,R_2],\ldots,R_n]_S$.

\begin{lem}\label{lemma2.3}
Let $R_1,\ldots,R_n$ be normal subgroups of $G$. Let $g_j\in R_j$ for $1\leq j\leq n$. Then
$$
\beta^n(g_{\sigma(1)},\ldots,g_{\sigma(n)})\in [[R_1,R_2],\ldots,R_n]_S
$$
for any $\sigma\in \Sigma_n$ and any bracket arrangement $\beta^n$ of weight $n$.
\end{lem}
\begin{proof}
The proof is given by double induction. The first induction is on $n$. Clearly the assertion holds for $n=1$. Suppose that the assertion holds for $m$ with $m<n$. Given an element $\beta^n(g_{\sigma(1)},\ldots,g_{\sigma(n)})$ as in the lemma. Then
$$
\beta^n(g_{\sigma(1)},\ldots,g_{\sigma(n)})=[\beta^p(g_{\sigma(1)},\ldots,g_{\sigma(p)}),\beta^{n-p}(g_{\sigma(p+1)},\ldots,g_{\sigma(n)})]
$$
for some bracket arrangements $\beta^p$ and $\beta^{n-p}$ with $1\leq p\leq n-1$. The second induction is on $q=n-p$. If $q=1$, we have
$$
\beta^{n-1}(g_{\sigma(1)},\ldots,g_{\sigma(n-1)})\in [[R_{\sigma(1)},R_{\sigma(2)}],\ldots,R_{\sigma(n-1)}]_S
$$
by the first induction and so
$$
\begin{array}{rcl}
\beta^n(g_{\sigma(1)},\ldots,g_{\sigma(n)})&=&[\beta^{n-1}(g_{\sigma(1)},\ldots,g_{\sigma(n-1)}),g_{\sigma(n)}]\\
&\in& [[[R_{\sigma(1)},R_{\sigma(2)}],\ldots,R_{\sigma(n-1)}]_S, R_{\sigma(n)}]\\
\end{array}
$$
with
$$
\begin{array}{rl}
& [[[R_{\sigma(1)},R_{\sigma(2)}],\ldots,R_{\sigma(n-1)}]_S, R_{\sigma(n)}]\\
=&\left[\prod_{\tau\in\Sigma_{n-1}}[[R_{\sigma(\tau(1))},R_{\sigma(\tau(2))}],\ldots,R_{\sigma(\tau(n-1))}],R_{\sigma(n)}\right]\\
\rEq^{\textrm{by Lemma~\ref{lemma2.1}}}&\prod_{\tau\in\Sigma_{n-1}}[[[R_{\sigma(\tau(1))},R_{\sigma(\tau(2))}],\ldots,R_{\sigma(\tau(n-1))}], R_{\sigma(n)}]\\
\leq& [[R_1,R_2],\ldots,R_n]_S.\\
\end{array}
$$
Now suppose that the assertion holds for $q'=n-p<q$. By the first induction, we have
$$
\beta^p(g_{\sigma(1)},\ldots,g_{\sigma(p)})\in [[R_{\sigma(1)},R_{\sigma(2)}],\ldots,R_{\sigma(p)}]_S
$$
and
$$
\beta^{n-p}(g_{\sigma(p+1)},\ldots,g_{\sigma(n)})\in [[R_{\sigma(p+1)},R_{\sigma(p+2)}],\ldots,R_{\sigma(n)}]_S.
$$
Thus
$$
\beta^{n}(g_{\sigma(1)},\ldots,g_{\sigma(n)})\in \left[[[R_{\sigma(1)},R_{\sigma(2)}],\ldots,R_{\sigma(p)}]_S, [[R_{\sigma(p+1)},R_{\sigma(p+2)}],\ldots,R_{\sigma(n)}]_S\right].
$$
By Lemma~\ref{lemma2.1}, $\beta^{n}(g_{\sigma(1)},\ldots,g_{\sigma(n)})$ lies in the product subgroup
$$
T=\prod\limits_{
\begin{array}{c}
\tau\in\Sigma_p\\
\rho\in\Sigma_{n-p}\\
\end{array}}
\left[[[R_{\sigma(\tau(1))},\ldots,R_{\sigma(\tau(p))}], [[R_{\sigma(\rho(p+1))},\ldots,R_{\sigma(\rho(n))}]\right],
$$
where $\Sigma_{n-p}$ acts on $\{p+1,\ldots,n\}$. By applying Hall's Theorem, we have
$$
\begin{array}{rl}
&\left[[[R_{\sigma(\tau(1))},\ldots,R_{\sigma(\tau(p))}], [[R_{\sigma(\rho(p+1))},\ldots,R_{\sigma(\rho(n))}]\right]\\
=&\left[[[R_{\sigma(\tau(1))},\ldots,R_{\sigma(\tau(p))}], \left[[[R_{\sigma(\rho(p+1))},\ldots, R_{\sigma(\rho(n-1))}],R_{\sigma(\rho(n))}\right]\right]\\
\leq &\left[\left[[[R_{\sigma(\tau(1))},\ldots,R_{\sigma(\tau(p))}], [[R_{\sigma(\rho(p+1))},\ldots, R_{\sigma(\rho(n-1))}]\right],R_{\sigma(\rho(n))}\right]\\
&\cdot\left[ \left[[[R_{\sigma(\tau(1))},\ldots,R_{\sigma(\tau(p))}],R_{\sigma(\rho(n))}\right], [[R_{\sigma(\rho(p+1))},\ldots, R_{\sigma(\rho(n-1))}]\right].\\
\end{array}
$$
Note that $A=\left[\left[[[R_{\sigma(\tau(1))},\ldots,R_{\sigma(\tau(p))}], [[R_{\sigma(\rho(p+1))},\ldots, R_{\sigma(\rho(n-1))}]\right],R_{\sigma(\rho(n))}\right]$ is generated by the elements of the form
$$
\left[\left[[[g'_{\sigma(\tau(1))},\ldots,g'_{\sigma(\tau(p))}], [[g'_{\sigma(\rho(p+1))},\ldots, g'_{\sigma(\rho(n-1))}]\right],g'_{\sigma(\rho(n))}\right]
$$
with $g'_j\in R_j$. By the second induction on case that $q=1$, the above elements lie in
$[[R_1,R_2],\ldots,R_n]_S$ and so
$$
A\leq [[R_1,R_2],\ldots,R_n]_S.
$$
Similarly, by the second induction hypothesis, $$\left[ \left[[[R_{\sigma(\tau(1))},\ldots,R_{\sigma(\tau(p))}],R_{\sigma(\rho(n))}\right], [[R_{\sigma(\rho(p+1))},\ldots, R_{\sigma(\rho(n-1))}]\right]$$ is a subgroup of $[[R_1,R_2],\ldots,R_n]_S$. It follows that
$$
T\leq [[R_1,R_2],\ldots,R_n]_S
$$
and so
$$
\beta^{n}(g_{\sigma(1)},\ldots,g_{\sigma(n)})\in [[R_1,R_2],\ldots,R_n]_S.
$$
Both the first and second inductions are finished, hence the result.
\end{proof}

\begin{lem}\label{lemma2.4}
Let $G$ be a group and let $R_1,\ldots,R_n$ be normal subgroups of $G$. Let $(i_1,i_2,\ldots,i_p)$ be a sequence of integers with $1\leq i_s\leq n$. Suppose that $$\{i_1,i_2,\ldots,i_p\}=\{1,2,\ldots,n\}.$$ Then
$$
[[R_{i_1},R_{i_2}],\ldots,R_{i_p}]\leq [[R_1,R_2],\ldots,R_n]_S.
$$
\end{lem}
\begin{proof}The proof is given by double induction. The first induction is on  $n$.
 The assertion clearly holds for $n=1$. Suppose that the assertion holds for $n-1$ with $n>1$.
From the hypothesis that $\{i_1,i_2,\ldots,i_p\}=\{1,2,\ldots,n\}$, we have $p\geq n$. When $p=n$, $(i_1,\ldots,i_n)$ is a permutation of $(1,\ldots,n)$ and so $$[[R_{i_1},R_{i_2}],\ldots,R_{i_n}]\leq [[R_1,R_2],\ldots,R_n]_S.$$ Suppose that
$$
[[R_{j_1},R_{j_2}],\ldots,R_{j_q}]\leq [[R_1,R_2],\ldots,R_n]_S
$$
for any sequence $(j_1,\ldots,j_q)$ with $q<p$ and $\{j_1,\ldots,j_q\}=\{1,\ldots,n\}$. Let $(i_1,\ldots,i_p)$ be a sequence with $\{i_1,\ldots,i_p\}=\{1,\ldots,n\}$.

If $i_p\in \{i_1,\ldots,i_{p-1}\}$, then $\{i_1,\ldots,i_{p-1}\}=\{1,\ldots,n\}$ and so
$$
[[R_{i_1},R_{i_2}],\ldots,R_{i_{p-1}}]\leq [[R_1,R_2],\ldots,R_n]_S
$$
by the second induction hypothesis. It follows that
$$
[[R_{i_1},R_{i_2}],\ldots,R_{i_p}]\leq [[R_1,R_2],\ldots,R_n]_S.
$$

If $i_p\not\in \{i_1,\ldots,i_{p-1}\}$ , we may assume that $i_p=n$. Then $$\{i_1,\ldots,i_{p-1}\}=\{1,\ldots,n-1\}$$ and so
$$
[[R_{i_1},R_{i_2}],\ldots,R_{i_{p-1}}]\leq [[R_1,R_2],\ldots,R_{n-1}]_S
$$
by the first induction hypothesis. From Lemma~\ref{lemma2.3}, we have
$$
[[R_{i_1},R_{i_2}],\ldots,R_{i_p}]\leq [[R_1,R_2],\ldots,R_n]_S.
$$
The inductions are  finished, hence the result holds.
\end{proof}

\begin{lem}\label{lemma2.5}
Let $G$ be a group and let $R_1,\ldots,R_n$ be normal subgroups of $G$ with $n\geq 2$. Let $(i_1,\ldots,i_p)$ and $(j_1,\ldots,j_q)$ be sequences of integers such that $\{i_1,\ldots,i_p\}\cup\{j_1,\ldots,j_q\}=\{1,2,\ldots,n\}$. Then
$$
[[[R_{i_1},R_{i_2}],\ldots,R_{i_p}],[[R_{j_1},R_{j_2}],\ldots,R_{j_q}]]\leq [[R_1,R_2],\ldots,R_n]_S.
$$
\end{lem}
\begin{proof}
The proof is given by double induction on $n$ and $q$ with $n\geq 2$ and $q\geq 1$. First we prove the assertion holds for $n=2$. If  $\{i_1,\ldots,i_p\}=\{1,2\}$
or $\{j_1,\ldots,j_q\}=\{1,2\}$, we have
$$
[[R_{i_1},R_{i_2}],\ldots,R_{i_p}]\leq [[R_1,R_2]_S \textrm{ or }[[R_{j_1},R_{j_2}],\ldots,R_{j_q}]\leq [[R_1,R_2]_S$$
by Lemma~\ref{lemma2.4}  and so  $$
[[[R_{i_1},R_{i_2}],\ldots,R_{i_p}],[[R_{j_1},R_{j_2}],\ldots,R_{j_q}]]\leq [[R_1,R_2]_S.
$$
Otherwise, $i_1=\cdots=i_p$ and $j_1=\cdots=j_q $, since  $\{i_1,\ldots,i_p\}\cup\{j_1,\ldots,j_q\}=\{1,2\}$, we may assume that  $i_1=\cdots=i_p=1, j_1=\cdots=j_q =2$, then $$[[R_{i_1},R_{i_2}],\ldots,R_{i_p}]\leq R_1 \textrm{ and }[[R_{j_1},R_{j_2}],\ldots,R_{j_q}]\leq R_2$$ and so
 $$
[[[R_{i_1},R_{i_2}],\ldots,R_{i_p}],[[R_{j_1},R_{j_2}],\ldots,R_{j_q}]]\leq [[R_1,R_2]_S.
$$

Suppose the assertion holds for $n-1$, that is $$
[[[R_{i_1},R_{i_2}],\ldots,R_{i_p}],[[R_{j_1},R_{j_2}],\ldots,R_{j_q}]]\leq [[R_1,R_2],\ldots,R_{n-1}]_S.
$$ when $\{i_1,\ldots,i_p\}\cup\{j_1,\ldots,j_q\}=\{1,2,\ldots,n-1\}$. We will use the second induction on $q$ to prove the assertion holds for $n$.

When $q=1$, the assertion follows by Lemma~\ref{lemma2.4}. Suppose that the assertion holds for $q-1$. By the Hall Theorem, $[[[R_{i_1},R_{i_2}],\ldots,R_{i_p}],[[R_{j_1},R_{j_2}],\ldots,R_{j_q}]]$ is a subgroup of the product
$$
[[[[R_{i_1},\ldots,R_{i_p}],[[R_{j_1},\ldots,R_{j_{q-1}}]],R_{j_q}]\cdot [[[[R_{i_1},\ldots,R_{i_p}],R_{j_q}],[[R_{j_1},\ldots,R_{j_{q-1}}]].
$$

By the second induction we have $$ [[[[R_{i_1},\ldots,R_{i_p}],R_{j_q}],[[R_{j_1},\ldots,R_{j_{q-1}}]]\leq [[R_1,R_2],\ldots,R_n]_S.$$

If $\{i_1,\ldots,i_p\}\cup\{j_1,\ldots,j_{q-1}\}=\{1,2,\ldots,n\}$, by the second induction $$[[[R_{i_1},\ldots,R_{i_p}],[[R_{j_1},\ldots,R_{j_{q-1}}]]\leq [[R_1,R_2],\ldots,R_n]_S$$ and hence $$[[[[R_{i_1},\ldots,R_{i_p}],[[R_{j_1},\ldots,R_{j_{q-1}}]],R_{j_q}]\leq [[R_1,R_2],\ldots,R_n]_S. $$

If $\{i_1,\ldots,i_p\}\cup\{j_1,\ldots,j_{q-1}\}\neq \{1,2,\ldots,n\}$, we may assume that  $$\{i_1,\ldots,i_p\}\cup\{j_1,\ldots,j_{q-1}\}=\{1,2,\ldots,n-1\}$$ and $j_q=n$. By the first induction, $$[[[R_{i_1},\ldots,R_{i_p}],[[R_{j_1},\ldots,R_{j_{q-1}}]]\leq [[R_1,R_2],\ldots,R_{n-1}]_S.$$
Then $$[[[[R_{i_1},\ldots,R_{i_p}],[[R_{j_1},\ldots,R_{j_{q-1}}]],R_{j_q}]\leq [[R_1,R_2],\ldots,R_n]_S.$$

It follows that $[[[R_{i_1},R_{i_2}],\ldots,R_{i_p}],[[R_{j_1},R_{j_2}],\ldots,R_{j_q}]]\leq [[R_1,R_2],\ldots,R_n]_S$. The double induction is finished, hence the result.
\end{proof}

\subsection{Proof of Theorem~\ref{theorem1.1}}
Clearly $[[R_1,R_2],\ldots,R_n]_S\leq [[R_1,R_2,\ldots, R_n]]$. We prove by induction on $n$ that $$[[R_1,R_2,\ldots,R_n]]\leq [[R_1,R_2],\ldots,R_n]_S.$$ The assertion holds for $n=1$.
\begin{hypothesis}\label{induction1}
Suppose that
$$
[[R_1,R_2,\ldots, R_s]]\leq [[R_1,R_2],\ldots, R_s]_S
$$
for any normal subgroups $R_1,\ldots,R_s$ of $G$ with $1\leq s<n$.
\end{hypothesis}

Let $R_1,\ldots, R_n$ be any normal subgroups of $G$. By definition, $[[R_1,R_2,\ldots,R_n]]$ is generated by all commutators $$\beta^t(g_{i_1},\ldots,g_{i_t})$$ of weight $t$ such that
$\{i_1,i_2,\ldots,i_t\}=\{1,2,\ldots,n\}$ with $g_j\in R_j$. To prove that each generator $\beta^t(g_{i_1},\ldots,g_{i_t})\in [[R_1,R_2],\ldots,R_n]_S$, we start the second induction on the weight $t$ of $\beta^t$ with $t\geq n$. If $t=n$, then $(i_1,\ldots,i_n)$ is a permutation of $(1,\ldots,n)$ and so the assertion holds by Lemma~\ref{lemma2.3}. Now assume that the following hypothesis holds:

\begin{hypothesis}\label{induction2}
Let $n\leq k<t$ and let
$$
\beta^k(g'_{i_1},\ldots,g'_{i_k})
$$
be any bracket arrangement of weight $k$ such that
where
\begin{enumerate}
\item[1)] $1\leq i_s\leq n$;
\item[2)] $\{i_1,\ldots,i_k\}=\{1,\ldots,n\}$;
\item[3)] $g'_j\in R_j$;
\end{enumerate}
Then $\beta^k(g'_{i_1},\ldots,g'_{i_k})\in [[R_1,R_2],\ldots,R_n]_S$.
\end{hypothesis}

Let $\beta^t(g_{i_1},\ldots,g_{i_t})$ be any bracket arrangement of weight $t$ with $\{i_1,\ldots,i_t\}=\{1,\ldots,n\}$ and $g_j\in R_j$ for $1\leq j\leq n$. From the definition of bracket arrangement, we have
$$
\beta^t(g_{i_1},\ldots,g_{i_t})=[\beta^p(g_{i_1},\ldots,g_{i_p}), \beta^{t-p}(g_{i_{p+1}},\ldots,g_{i_t})]
$$
for some bracket arrangements $\beta^p$ and $\beta^{t-p}$ of weight $p$ and $t-p$, respectively, with $1\leq p\leq n-1$. Let
$$
A=\{i_1,\ldots,i_p\}\textrm{ and } B=\{i_{p+1},\ldots,i_t\}.
$$
Then both $A$ and $B$ are the subsets of $\{1,\ldots,n\}$ with $A\cup B=\{1,\ldots,n\}$.

Suppose that the cardinality $|A|=n$ or $|B|=n$. We may assume that $|A|=n$. By Hypothesis~\ref{induction2},
$$
\beta^p(g_{i_1},\ldots,g_{i_p})\in [[R_1,R_2],\ldots,R_n]_S.
$$
Since $[[R_1,R_2],\ldots,R_n]_S$ is a normal subgroup of $G$, we have
$$
\beta^t(g_{i_1},\ldots,g_{i_t})=[\beta^p(g_{i_1},\ldots,g_{i_p}), \beta^{t-p}(g_{i_{p+1}},\ldots,g_{i_t})]\in [[R_1,R_2],\ldots,R_n]_S.
$$
This proves the result in this case.

Suppose that $|A|<n$ and $|B|<n$. Let $A=\{l_1,\ldots,l_a\}$ with $1\leq l_1<l_2<\cdots<l_a\leq n$ and $1\leq a<n$, and let $B=\{k_1,\ldots,k_b\}$ with $1\leq k_1<k_2<\cdots< k_b$ and $1\leq b<n$. Observe that
$$
\beta^p(g_{i_1},\ldots,g_{i_p})\in [[R_{l_1},R_{l_2},\ldots,R_{l_a}]].
$$
By Hypothesis~\ref{induction1},
$$
[[R_{l_1},R_{l_2},\ldots,R_{l_a}]]=[[R_{l_1},R_{l_2}],\ldots,R_{l_a}]_S.
$$
Thus
$$
\beta^p(g_{i_1},\ldots,g_{i_p})\in [[R_{l_1},R_{l_2}],\ldots,R_{l_a}]_S.
$$
Similarly
$$
\beta^{t-p}(g_{i_{p+1}},\ldots,g_{i_t})\in [[R_{k_1},R_{k_2}],\ldots,R_{k_b}]_S.
$$
It follows that the element
$\beta^t(g_{i_1},\ldots,g_{i_t})$ lies in the commutator subgroup
$$
\left[[[R_{l_1},R_{l_2}],\ldots,R_{l_a}]_S, [[R_{k_1},R_{k_2}],\ldots,R_{k_b}]_S\right].
$$
From Lemma~\ref{lemma2.5}, we have
$$
\left[[[R_{l_{\sigma(1)}},R_{l_{\sigma(2)}}],\ldots,R_{l_{\sigma(a)}}], [[R_{k_{\tau(1)}},R_{k_{\tau(2)}}],\ldots,R_{k_{\tau(b)}}]\right]\leq [[R_1,R_2],\ldots,R_n]_S
$$
for all $\sigma\in\Sigma_a$ and $\tau\in\Sigma_b$ because $\{l_1,\ldots,l_a\}\cup \{k_1,\ldots,k_b\}=A\cup B=\{1,2,\ldots,n\}$. It follows from Lemma~\ref{lemma2.1} that
$$
\left[[[R_{l_1},R_{l_2}],\ldots,R_{l_a}]_S, [[R_{k_1},R_{k_2}],\ldots,R_{k_b}]_S\right]\leq [[R_1,R_2],\ldots,R_n]_S.
$$
Thus
$$
\beta^t(g_{i_1},\ldots,g_{i_t})\in [[R_1,R_2],\ldots,R_n]_S.
$$
The inductions are finished, hence Theorem~\ref{theorem1.1}.

\subsection{Proof of Theorem~\ref{theorem1.2}}
Clearly $$\prod_{\sigma\in \Sigma_{n-1}}[[R_{1},R_{\sigma(2)}],\ldots,R_{\sigma(n)}]\leq [[R_1,R_2],\ldots,R_n]_S.$$
We prove by induction on $n$ that $$[[R_1,R_2],\ldots,R_n]_S\leq \prod_{\sigma\in \Sigma_{n-1}}[[R_{1},R_{\sigma(2)}],\ldots,R_{\sigma(n)}].$$ The assertion holds for $n=1$.

Suppose that
$$
[[R_1,R_2],\ldots, R_s]_S\leq \prod_{\sigma\in \Sigma_{s-1}}[[R_{1},R_{\sigma(2)}],\ldots,R_{\sigma(s)}]
$$
for any normal subgroups $R_1,\ldots,R_s$ of $G$ with $1\leq s<n$.

Let $R_1,\ldots, R_n$ be any normal subgroups of $G$. By definition, $$[[R_1,R_2],\ldots,R_n]_S=\prod\limits_{\tau\in \Sigma_{n}}[[R_{\tau(1)},R_{\tau(2)}],\ldots,R_{\tau(n)}].$$ It suffices to prove that for any $\tau\in \Sigma_n$, $$[[R_{\tau(1)},R_{\tau(2)}],\ldots,R_{\tau(n)}]\leq \prod_{\sigma\in \Sigma_{n-1}}[[R_{1},R_{\sigma(2)}],\ldots,R_{\sigma(n)}].$$
The assertion holds for  $\tau(1)=1$.
Suppose that the assertion holds for $\tau(k-1)=1$ with $1\leq k-1<n$.
When $\tau(k)=1$, consider the subgroup $$[[[[[R_{\tau(1)},\ldots,R_{\tau(k-2)}],R_{\tau(k-1)}],R_1],R_{\tau(k+1)},\ldots,R_{\tau(n)}].$$
Following from the Hall Theorem,
$$
\begin{array}{rl}
&[[[[[R_{\tau(1)},\ldots, R_{\tau(k-2)}],R_{\tau(k-1)}],R_1],R_{\tau(k+1)},\ldots,R_{\tau(n)}]\\
&\\
\leq &[[[[[R_{\tau(1)},\ldots,R_{\tau(k-2)}],R_1],R_{\tau(k-1)}], R_{\tau(k+1)},\ldots,R_{\tau(n)}]]\\
&\cdot[   [[R_{\tau{(k-1)}},R_1], [[R_{\tau(1)},\ldots, R_{\tau(k-2)}]],   R_{\tau(k+1)},\ldots,R_{\tau(n)}                                                 ].\\
\end{array}
$$
By the second induction $$
\begin{array}{rl}
&[[[[[R_{\tau(1)},\ldots,R_{\tau(k-2)}],R_1],R_{\tau(k-1)}], R_{\tau(k+1)},\ldots,R_{\tau(n)}]]\\
\leq &\prod\limits_{\sigma\in \Sigma_{n-1}}[[R_{1},R_{\sigma(2)}],\ldots,R_{\sigma(n)}].\\
\end{array}
$$
From Lemma~\ref{lemma2.5} and the first induction, we have
$$
\begin{array}{rcl}
[[R_{\tau{(k-1)}},R_1], [[R_{\tau(1)},\ldots, R_{\tau(k-2)}]]&\leq& [[R_1,R_{l_1}],R_{l_2},\ldots,R_{l_{k-1}}]_S\\
&\leq& \prod\limits_{\sigma\in \Sigma_{k-1}}[[R_{1},R_{\sigma(l_1)}],\ldots,R_{\sigma(l_{k-1})}],\\
\end{array} $$
where
$\{1,\tau(1),\ldots,\tau{(k-1)}\}=\{1,l_1,l_2,\ldots,l_{k-1}\}$ with $1<l_1<l_2<\cdots<l_{k-1}$.
By Lemma~\ref{lemma2.1}
$$
\begin{array}{rl}
&[  [[R_{\tau{(k-1)}},R_1], [[R_{\tau(1)},\ldots, R_{\tau(k-2)}]],   R_{\tau(k+1)},\ldots,R_{\tau(n)}                                                 ]\\
&\\
\leq &\prod\limits_{\sigma\in \Sigma_{k-1}}[ [[R_{1},R_{\sigma(l_1)}],\ldots,R_{\sigma(l_{k-1})}],R_{\tau(k+1)},\ldots,R_{\tau(n)}                                                 ]\\
&\\
\leq &\prod\limits_{\sigma\in \Sigma_{n-1}}[[R_{1},R_{\sigma(2)}],\ldots,R_{\sigma(n)}]\\
\end{array}
$$
because for each $\sigma\in \Sigma_{k-1}$, the sequence $(\sigma(l_1),\ldots,\sigma(l_{k-1}),\tau(k+1),\ldots,\tau(n))$ is a permutation of $(\tau(1),\ldots,\tau(k-1),\tau(k+1),\ldots,\tau(n))$ which is a permutation of $(2,\ldots,n)$.
 It follows that
 $$
 \begin{array}{rl}
 &[[[[[R_{\tau(1)},\ldots,R_{\tau(k-2)}],R_{\tau(k-1)}],R_1],R_{\tau(k+1)},\ldots,R_{\tau(n)}]\\
&\\
 \leq &\prod\limits_{\sigma\in \Sigma_{n-1}}[[R_{1},R_{\sigma(2)}],\ldots,R_{\sigma(n)}]\\
 \end{array}
 $$
and so $$[[R_1,R_2],\ldots,R_n]_S\leq \prod_{\sigma\in \Sigma_{n-1}}[[R_{1},R_{\sigma(2)}],\ldots,R_{\sigma(n)}].$$ This finishes the proof.

\section{Proof of Theorem~\ref{theorem1.3}}\label{section3}
\subsection{Ellis-Mikhailov Theorem.}
In this subsection, we review some terminology and the main result in~\cite{EM}. Let $G$ be a group. An $m$-tuple of normal subgroups $(R_1,\ldots,R_m)$ of $G$ is called \textit{connected} if either
\begin{enumerate}
\item  $m\leq 2$ or
\item $m\geq 3$ with the property that: for all subsets $I, J \subseteq \{1,\cdots,m\}$ with $ |I|\geq 2,
|J|\geq 1$
\begin{equation}\label{connectivity}
\left(\bigcap_{i\in I} R_i \right)\cdot  \prod_{j\in J}R_j =
\bigcap_{i\in I} \left( R_i\cdot \prod_{j\in J}R_j \right).
\end{equation}
\end{enumerate}

Let  $G$  be a group with normal subgroups $R_1,\ldots,R_n$. Let $X(G;R_1,\ldots,R_n)$ be the homotopy colimit of the cubical diagram obtained from  classifying spaces $B(G/\prod_{i\in I} R_i)$ with the maps
$$
B(G/\prod_{i\in I} R_i)\to B(G/\prod_{i'\in I'}R_{i'})
$$
induced by the canonical quotient homomorphism $G/\prod_{i\in I} R_i \twoheadrightarrow G/\prod_{i'\in I'}R_{i'}$ for $I\subseteq I'$, where $I$ ranges over all proper subsets $I \subsetneq \{1,\ldots,n\}$.

In the following theorem and the rest of the article, the notation $(\ \cdots \ \hat a \ \cdots)$ means to remove letter $a$.

\begin{thm}~\cite[Theorem 1]{EM}\label{theorem3.1}
Let  $G$  be a group with normal subgroups $R_1,\ldots,R_n$ with $n\geq2$. Let $X=X(G;R_1,\ldots,R_n)$.
Suppose that the $(n-1)$-tuple $$(R_1,\ldots,\hat R_i,\ldots,R_n)$$ is connected for each $1\leq i\leq n$. Then
$$
\pi_n(X) \cong \frac{R_1\cap \cdots \cap R_n}{\prod_{I\cup J=\{1,\ldots,n\},
 {I\cap J=\emptyset}}
[\cap_{i\in I}R_i,\cap_{j\in J}R_j]}.
$$\hfill $\Box$
\end{thm}

\subsection{Proof of Theorem~\ref{theorem1.3}}

Theorem~\ref{theorem1.3} is part of the following statement.

\begin{thm}
Let $(X,A)$ be a pair of spaces and let $(A_1,\ldots,A_n)$ be a cofibrant $n$-partition of $X$ relative to $A$ with $n\geq 2$. Suppose that
\begin{enumerate}
\item[i)] For any proper subset $I=\{i_1,\ldots,i_k\}\subsetneq \{1,2,\ldots,n\}$, the union
$
\bigcup_{i\in I}A_i
$
is a path-connected $K(\pi,1)$-space.
\item[ii)] The inclusion $A \to A_i$ induces an epimorphism of the corresponding fundamental groups for each $1\leq i\leq n$.
\end{enumerate}
Let $R_i$ be the kernel of $\pi_1(A)\to \pi_1(A_i)$ for $1\leq i\leq n$. Then
\begin{enumerate}
\item For any proper subset $I=\{i_1,\ldots,i_k\}\subsetneq \{1,2,\ldots,n\}$,
$$
R_{i_1}\cap\cdots\cap R_{i_k}=[[R_{i_1}, R_{i_2}],\ldots,R_{i_k}]_S.
$$
\item The $(n-1)$-tuple $(R_1,\ldots,\hat R_i,\ldots,R_n)$ is connected for each $1\leq i\leq n$.
\item There is an isomorphism of groups
$$
\pi_n(X)\cong (R_1\cap R_2\cap \cdots\cap R_n)/[[R_1,R_2],\ldots,R_n]_S.
$$
\item For $1< k< n$, $I=\{i_1,\ldots,i_k\}\subsetneq \{1,2,\ldots,n\}$ and $J= \{1,2,\ldots,n\}\smallsetminus I$, there is an isomorphism
$$
\pi_k(X)\cong \left.\left(\bigcap_{s=1}^k(R_{i_s}\cdot \prod_{j\in J}R_j)\right)\right/\left(([[R_{i_1}, R_{i_2}],\ldots,R_{i_k}]_S)\cdot \prod_{j\in J}R_j \right).
$$
\end{enumerate}
\end{thm}
\begin{proof}
We prove assertions (1)-(3) by induction on $n$. If $n=2$, assertions (1) and (2) hold obviously and assertion (3) follows from the classical Brown-Loday Theorem~\cite{BL}.

 Suppose that assertions (1)-(3) hold for all cofibrant $m$-partitions $(B_1,\ldots,B_m)$ of any space $Y$ relative to $B$ satisfying conditions (i) and (ii) with $m<n$.

Let $(X,A)$ be a pair of spaces and let $(A_1,\ldots,A_n)$ be a cofibrant $n$-partition of $X$ relative to $A$ with $n>2$.

(1). Let $I=\{i_1,\ldots,i_k\}\subsetneq \{1,2,\ldots,n\}$ be a proper subset. If $k=1$, then $R_{i_1}=[R_{i_1}]_S$ by definition. We may assume that $2\leq k<n$. Let $B_s=A_{i_s}$ for $s=1,\ldots, k$ and let $Y=B_1\cup\cdots\cup B_{k}$. Then $(B_1,\ldots,B_k)$ is a cofibrant $k$-partition of $Y$ relative to $A$ satisfying conditions (i) and (ii). Notice that
$$
\Ker(\pi_1(A)\to \pi_1(B_s))=\Ker(\pi_1(A)\to \pi_1(A_{i_s}))=R_{i_s}.
$$
By induction hypothesis, there is an isomorphism
$$
\pi_k(Y)\cong (R_{i_1}\cap R_{i_2}\cap \cdots\cap R_{i_k})/[[R_{i_1},R_{i_2}],\ldots,R_{i_k}]_S.
$$
From condition (i), $Y=A_{i_1}\cup\cdots\cup A_{i_k}$ is a $K(\pi,1)$-space (as $k<n$) and so $\pi_k(Y)=0$ (as $k\geq2$). Thus
$$
R_{i_1}\cap R_{i_2}\cap \cdots\cap R_{i_k}=[[R_{i_1},R_{i_2}],\ldots,R_{i_k}]_S,
$$
which is assertion~(1).

(2). We show that the $(n-1)$-tuple $(R_1,\ldots,R_{n-1})$ is connected. Let $m=n-1$. If $m=2$, then $(R_1,R_2)$ is connected by definition. Assume that $m\geq 3$. Let $I, J \subseteq \{1,\cdots,m\}$ with $ |I|\geq 2$ and $|J|\geq 1$. We have to show the connectivity condition that
$$
\left(\bigcap_{i\in I} R_i \right)\cdot  \prod_{j\in J}R_j =
\bigcap_{i\in I} \left( R_i\cdot \prod_{j\in J}R_j \right).
$$
Clearly
$$
\left(\bigcap_{i\in I} R_i \right)\cdot \prod_{j\in J}R_j\leq
\bigcap_{i\in I} \left( R_i\cdot \prod_{j\in J}R_j \right).
$$
Now we show the other direction. Let $I=\{i_1,\ldots,i_k\}$ with $2\leq k<n$. Let $B=\bigcup_{j\in J}A_j$ and let $B_s=B\cup A_{i_s}$ for $s=1,2,\ldots,k$. Let
$$
Y=\bigcup_{s=1}^kB_s=\bigcup_{i\in I, j\in J}A_i\cup A_j.
$$
Then $(B_1,\ldots,B_k)$ is a cofibrant $k$-partition of $Y$ relative to $B$. Clearly condition (i) holds. Let $G=\pi_1(A)$. Then $\pi_1(B)=G/\prod_{j\in J}R_j$ and
$$
\pi_1(B_s)=G/(R_{i_s}\cdot\prod_{j\in J}R_j).
$$
Thus condition (ii) holds. Let
$$
\begin{array}{rcl}
N_s&=&\Ker(\pi_1(B)\to \pi_1(B_s))\\
&=&\Ker(G/(\prod_{j\in J}R_j)\to G/(R_{i_s}\cdot\prod_{j\in J}R_j))\\
&=&(R_{i_s}\cdot\prod_{j\in J}R_j)/\prod_{j\in J}R_j\\
&=&R_{i_s}/(R_{i_s}\cap \prod_{j\in J}R_j).\\
\end{array}
$$
From induction hypothesis, we have
$$
\pi_k(Y)=(N_1\cap N_2\cap\cdots\cap N_k)/[[N_1,N_2],\ldots,N_k]_S.
$$
Since $
Y=\bigcup_{i\in I, j\in J}A_i\cup A_j$ with $I,J\subseteq\{1,2,\ldots,n-1\}$, the space $Y$ is a $K(\pi,1)$-space from condition (i). Thus $\pi_k(Y)=0$ and so
\begin{equation}\label{equation3.2}
N_1\cap N_2\cap\cdots\cap N_k=[[N_1,N_2],\ldots,N_k]_S
\end{equation}
with $N_s=R_{i_s}/(R_{i_s}\cap \prod_{j\in J}R_j)$ in $G/\prod_{j\in J}R_j$. Consider the quotient homomorphism
$$
\phi\colon G\longrightarrow G/\prod_{j\in J}R_j.
$$
Then
$$
N_s=R_{i_s}/(R_{i_s}\cap \prod_{j\in J}R_j)=\phi(R_{i_s})
$$
and so
$$
\begin{array}{rcl}
\phi\left([[R_{i_1},R_{i_2}],\ldots,R_{i_k}]_S\right)&=&[[\phi(R_{i_1}),\phi(R_{i_2})],\ldots,\phi(R_{i_k})]_S\\
&=&[[N_1,N_2],\ldots,N_k]_S\\
&=&N_1\cap N_2\cap\cdots\cap N_k \textrm{ by equation~(\ref{equation3.2})}.\\
\end{array}
$$
From the fact that
$$
\phi\left(\bigcap_{i\in I} \left( R_i\cdot \prod_{j\in J}R_j \right)\right)\leq \bigcap_{s=1}^k N_s
$$
together with equation~(\ref{equation3.2}),
we have
$$
\begin{array}{rcl}
\bigcap_{i\in I} \left( R_i\cdot \prod_{j\in J}R_j \right)&\leq&\phi^{-1}\left(\phi\left([[R_{i_1},R_{i_2}],\ldots,R_{i_k}]_S\right)\right)\\ &=&[[R_{i_1},R_{i_2}],\ldots,R_{i_k}]_S\cdot \prod_{j\in J}R_j\\
&\leq& \left(\bigcap_{s=1}^k R_{i_s}\right)\cdot\prod_{j\in J} R_j\\
&=&\left(\bigcap_{i\in I}R_i\right)\cdot \prod_{j\in J} R_j.\\
\end{array}
$$
This proves that $(R_1,\ldots,R_{n-1})$ is connected. Similarly, each $(R_1,\ldots, \hat{R}_i,\ldots, R_n)$ is connected for $1\leq i< n$ and
hence assertion (2).

(3). From assertion (2), each $(R_1,\ldots, \hat{R}_i,\ldots, R_n)$ is connected for $1\leq i\leq n$. Since $(A_1,\ldots,A_n)$ is a cofibrant partition of $X$ relative to $A$, $X$ is the homotopy colimit of the diagram given by the inclusions
$$
A\subseteq A_{j_1}\cup A_{j_2}\cup\cdots\cup A_{j_q}\subseteq A_{i_1}\cup A_{i_2}\cup\cdots\cup A_{i_p}
$$
with $\{j_1,\ldots,j_q\}\subseteq \{i_1,\ldots,i_p\}\subsetneq \{1,2,\ldots,n\}$. From Van Kampen's theorem
$$
\pi_1(A_{i_1}\cup \cdots\cup A_{i_k})=\pi_1(A)/\prod_{s=1}^k R_{i_s}.
$$
Thus $X=X(\pi_1(A);R_1,R_2,\ldots,R_n)$. From Theorem~\ref{theorem3.1}, we have
$$
\pi_n(X)\cong \frac{R_1\cap \cdots \cap R_n}{\prod_{I\cup J=\{1,\ldots,n\},{I\cap J=\emptyset}}
[\cap_{i\in I}R_i,\cap_{j\in J}R_j]}.
$$
It suffices to show that
$$
\prod_{I\cup J=\{1,\ldots,n\}, {I\cap J=\emptyset}}
\left[\bigcap_{i\in I}R_i,\bigcap_{j\in J}R_j\right]=[[R_1,R_2],\ldots,R_n]_S.
$$

Recall that $$[[R_1,R_2],\ldots,R_n]_S=\prod_{\sigma\in \Sigma_n}[[R_{\sigma(1)},R_{\sigma(2)}],\ldots,R_{\sigma(n)}].$$
For each $\sigma\in\Sigma_n$, let $I=\{\sigma(1),\ldots,\sigma(n-1)\}, J=\{\sigma(n)\}$, then
$$
\begin{array}{rcl}
[[R_{\sigma(1)},R_{\sigma(2)}],\ldots,R_{\sigma(n)}]&=&[[[R_{\sigma(1)},R_{\sigma(2)}],\ldots,R_{\sigma(n-1)}],R_{\sigma(n)}]\\
&&\\
&\leq &\left[\bigcap_{i\in I}R_i,\bigcap_{j\in J}R_j\right]\\
\end{array}
$$
and so
$$
[[R_1,R_2],\ldots,R_n]_S\leq \prod_{I\cup J=\{1,\ldots,n\}, {I\cap J=\emptyset}}
\left[\bigcap_{i\in I}R_i,\bigcap_{j\in J}R_j\right].
$$

Conversely let $I=\{i_1,\ldots,i_p\}$ and $J=\{j_1,\ldots,j_q\}$ with $1\leq p,q\leq n-1$, $I\cup J=\{1,\ldots,n\}$ and $I\cap J=\emptyset$. By assertion (1),
$$
\bigcap_{i\in I}R_i=[[R_{i_1},R_{i_2}],\ldots,R_{i_p}]_S
\textrm{
and }
\bigcap_{j\in J}R_j=[[R_{j_1},R_{j_2}],\ldots,R_{j_q}]_S.
$$
Thus
$$
\begin{array}{rcl}
\left[\bigcap_{i\in I}R_i,\bigcap_{j\in J}R_j\right]&=&\left[ [[R_{i_1},R_{i_2}],\ldots,R_{i_p}]_S,[[R_{j_1},R_{j_2}],\ldots,R_{j_q}]_S\right]\\
&\leq&[[R_1,R_2],\ldots,R_n]_S\\
\end{array}
$$
by Theorem~\ref{theorem1.1} because
$$
\left[ [[R_{i_1},R_{i_2}],\ldots,R_{i_p}]_S,[[R_{j_1},R_{j_2}],\ldots,R_{j_q}]_S\right]\leq [[R_1,R_2,\ldots,R_n]].
$$
This finishes the proof of assertion (3).

Assertion (4) follows from assertion (3) by constructing a new partition as follows: let $B=\bigcup_{j\in J}A_j$ and let $B_s=B\cup A_{i_s}$ for $s=1,2,\ldots,k$. Then
$$
X=\bigcup_{s=1}^kB_s=\bigcup_{i\in I, j\in J}A_i\cup A_j.
$$
and $(B_1,\ldots,B_k)$ is a cofibrant $k$-partition of $X$ relative to $B$. Condition (i) and (ii) hold similar as in the proof of assertion (2).  Let
$$
\begin{array}{rcl}
N_s&=&\Ker(\pi_1(B)\to \pi_1(B_s))\\
&=&R_{i_s}/(R_{i_s}\cap \prod_{j\in J}R_j).\\
\end{array}
$$
From assertion (3), we have $$
\pi_k(X)\cong \left(\bigcap_{s=1}^kN_s\right)/[[N_{1}, N_2],\ldots,N_k]_S.
$$
To finish the proof, it suffices to show that
\begin{equation}\label{equation3.3}
\begin{array}{rl}
 &\left(\bigcap_{s=1}^kN_s\right)/[[N_{1}, N_2],\ldots,N_k]_S\\
 \cong &(\bigcap_{s=1}^k(R_{i_s}\cdot \prod_{j\in J}R_j))/([[R_{i_1}, R_{i_2}],\ldots,R_{i_k}]_S\cdot \prod_{j\in J}R_j ).\\
 \end{array}
\end{equation}
Let $\phi\colon G\longrightarrow G/\prod_{j\in J}R_j$ be the quotient homomorphism. Then it is straightforward to check that
$$
\begin{array}{rcl}
\phi^{-1}(N_s)&=&R_{i_s}\cdot \prod_{j\in J}R_j,\\
&&\\
\phi^{-1}\left(\bigcap_{s=1}^kN_s\right)&=& \bigcap_{s=1}^k(R_{i_s}\cdot \prod_{j\in J}R_j),\\
&&\\
\phi^{-1}([[N_{1}, N_2],\ldots,N_k]_S)&=&[[R_{i_1}, R_{i_2}],\ldots,R_{i_k}]_S\cdot \prod_{j\in J}R_j.\\
\end{array}
$$
Thus equation~(\ref{equation3.3}) holds, hence assertion (4).

The proof is finished.
\end{proof}

\section{Applications to the Free Groups and Surface Groups}\label{section4}
\subsection{Subgroups of the Surface Groups}
Let $X=S$ be a path-connected compact $2$-manifold with or without boundary. Let $Q_i$ be a set of finite points in $S\smallsetminus \partial S$, $1\leq i\leq n$, such that
\begin{enumerate}
\item $Q_i\not=\emptyset$ for each $1\leq i\leq n$ and
\item $Q_i\cap Q_j=\emptyset$ for $i\not=j$.
\end{enumerate}
Let $
A=S\smallsetminus\left(\bigcup_{i=1}^n Q_i\right)
$
be the punctured surface and let $A_i=A\cup Q_i$. Then $(A_1,\ldots,A_n)$ is a cofibrant $n$-partition\footnote{One needs to do some modifications such that the cofibrant hypothesis holds: By replacing each punctured point by a small open disk in $S$, the resulting punctured surfaces become compact $2$-manifolds and so the cofibrant hypothesis for the partition $(A_1,\ldots,A_n)$ holds.}\label{footnote1} of $S$ relative to $A$.  For any subset $$\{i_1,\ldots,i_k\}\subsetneq \{1,2,\ldots,n\},$$ the space
$$
\bigcup_{s=1}^k A_{i_s}
$$
is a $K(\pi,1)$-space because it is a surface punctured  by at least
one point. Observe that each homomorphism
$$
\pi_1(A)\longrightarrow \pi_1(A_i)
$$
is an epimorphism. Let $R_i$ be the kernel of the epimorphism $\pi_1(A)\to \pi_1(A_i)$. By Theorem~\ref{theorem1.3}, we have
\begin{equation}\label{equation4.1}
\pi_n(S)\cong (R_1\cap R_2\cap\cdots\cap R_n)/[[R_1,R_2],\ldots,R_n]_S.
\end{equation}
We give a group theoretic interpretation of this isomorphism.

\bigskip

\noindent\textbf{Case 1.} $X=S^2$. Let $Q=\bigcup_{i=1}^nQ_i$ with
$$
Q=\{q_1,q_2,\ldots,q_m\}
$$
with a choice of order that $q_i<q_j$ for $i<j$. For each $q\in Q$, let $c_q$ be a generator in $\pi_1(A)$ represented by a small circle around the point $q$ with a choice of orientation such that $\pi_1(A)$ admits the presentation
$$
\pi_1(A)=\left\la c_q \ \left| \ q\in Q\ \prod_{j=1}^mc_{q_j}=1\right. \right\ra.
$$
Then $R_i=\la\la c_q \ | \ q\in Q_i\ra\ra$. From equation~(\ref{equation4.1}), we have the following result.

\begin{thm}\label{theorem4.1}
Let $G=\la x_1,x_2,\ldots,x_m \ | \ x_1x_2\cdots x_m=1\ra$ be the free group of rank $m-1$ with $m\geq2$. Let $n\geq 2$ and let $P_i$ be any subset of $\{x_1,x_2,\ldots,x_m\}$, $1\leq i\leq n$, such that
\begin{enumerate}
\item[(i)] $P_i\not=\emptyset$ for each $1\leq i\leq n$;
\item[(ii)] $P_i\cap P_j=\emptyset$ for $i\not=j$ and
\item[(iii)] $\bigcup_{i=1}^n P_i=\{x_1,x_2,\ldots,x_m\}$.
\end{enumerate}
Let $R_i=\la\la P_i\ra\ra$ be the normal closure of $P_i$ in $G$. Then there is an isomorphism of groups
$$
(R_1\cap R_2\cap \cdots\cap R_n)/[[R_1,R_2],\ldots,R_n]_S\cong\pi_n(S^2).
$$\hfill $\Box$
\end{thm}

If $n=m$ with $P_i=\{x_i\}$, then there is an isomorphism of groups
$$
(\la\la x_1\ra\ra\cap \la\la x_2\ra\ra\cap\cdots\cap \la\la x_n)/[[\la\la x_1\ra\ra,\la\la x_2\ra\ra], \ldots,\la\la x_n\ra\ra]_S\cong \pi_n(S^2),
$$
which is ~\cite[Theorem 1.7]{Wu1}. One interesting point of Theorem~\ref{theorem4.1} is that the factor group
$$
(R_1\cap R_2\cap\cdots \cap R_n)/[[R_1,R_2],\ldots,R_n]_S
$$
only depends on the length of the partition $(P_1,\ldots,P_n)$ of
$\{x_1,\ldots,x_m\}$, which does not seem obvious from the group
theoretic point of view.

\bigskip

\noindent\textbf{Case 2.} $X=\RP^2$. Let $Q=\bigcup_{i=1}^nQ_i$ with
$$
Q=\{q_1,q_2,\ldots,q_m\}.
$$
Then $\pi_1(A)$ admits a presentation
$$
\pi_1(A)=\left\la a_1, c_q \ \left| \ q\in Q,\ a_1^2=\prod_{i=1}^mc_{q_i}\right.\right\ra
$$
with $R_i=\la\la c_q \ | \ q\in Q_i\ra\ra$. Note that $\pi_n(\RP^2)\cong \pi_n(S^2)$ for $n\geq 2$. From equation~(\ref{equation4.1}), we have the following result.

\begin{thm}\label{theorem4.2}
Let $G=\la a_1, x_1,x_2,\ldots,x_m \ | \ a_1^2=x_1x_2\cdots x_m\ra$ with $m\geq 2$. Let $n\geq 2$ and let $P_i$ be any subset of $\{x_1,x_2,\ldots,x_m\}$, $1\leq i\leq n$, such that
\begin{enumerate}
\item[(i)] $P_i\not=\emptyset$ for each $1\leq i\leq n$;
\item[(ii)] $P_i\cap P_j=\emptyset$ for $i\not=j$ and
\item[(iii)] $\bigcup_{i=1}^n P_i=\{x_1,x_2,\ldots,x_m\}$.
\end{enumerate}
Let $R_i=\la\la P_i\ra\ra$ be the normal closure of $P_i$ in $G$. Then there is an isomorphism of groups
$$
(R_1\cap R_2\cap \cdots\cap R_n)/[[R_1,R_2],\ldots,R_n]_S\cong\pi_n(S^2).
$$\hfill $\Box$
\end{thm}

\noindent\textbf{Remark.} In the above theorem, $G\cong \pi_1(\RP^2\smallsetminus Q_m)$ is a free group of rank $m$ with the presentation given in the form as in the statement. Thus the subgroups $R_i$ are different from those given in Theorem~\ref{theorem4.1}.

\bigskip

\noindent\textbf{Case 3.} $X\not=S^2$ or $\RP^2$. Let $Q=\bigcup_{i=1}^nQ_i$ with
$$
Q=\{q_1,q_2,\ldots,q_m\}.
$$
If $X$ is an oriented surface of genus $g$ with $t$ boundary components, then $\pi_1(A)$ admits a presentation
$$
\pi_1(A)=\left\la a_1,b_1,\ldots,a_g,b_g, d_1,\ldots,d_t, c_q \ \left| \ q\in Q \ \prod_{i=1}^g[a_i,b_i]=\prod_{j=1}^td_j\cdot \prod_{i=1}^m c_{q_i}\right. \right\ra
$$
with $R_i=\la\la c_q \ | \ q\in Q_i\ra\ra$. If $X$ is a non-oriented surface of genus $h$ with $t$ boundary components, then $\pi_1(A)$ admits a presentation
$$
\pi_1(A)=\left\la a_1,\ldots,a_h, d_1,\ldots,d_t,c_q \ \left| \ q\in Q \ \prod_{i=1}^h a_i^2=\prod_{j=1}^td_j\cdot \prod_{i=1}^m c_{q_i}\right. \right\ra
$$
with $R_i=\la\la c_q \ | \ q\in Q_i\ra\ra$. Note that $\pi_n(X)=0$ for $n\geq 2$. From equation~(\ref{equation4.1}), we have the following result.

\begin{thm}\label{theorem4.3}
Let $G$ be one of the following surface groups:
\begin{enumerate}
\item[(a)] $\la a_1,b_1,\ldots,a_g,b_g, y_1,\ldots,y_t, x_1,\ldots,x_m \ | \ \prod_{i=1}^g[a_i,b_i]=\prod_{j=1}^ty_j\cdot \prod_{i=1}^m x_i\ra$
with $g>0$ or $t>0$.
\item[(b)]  $\la a_1,\ldots,a_h, y_1,\ldots,y_t, x_1,\ldots,x_m \ | \ \prod_{i=1}^h a_i^2=\prod_{j=1}^ty_j\cdot \prod_{i=1}^m x_i\ra$
with $h>1$ or $t>0$.
\end{enumerate}
Let $P_i$ be any subset of $\{x_1,x_2,\ldots,x_m\}$, $1\leq i\leq n$, such that
\begin{enumerate}
\item[(i)] $P_i\not=\emptyset$ for each $1\leq i\leq n$;
\item[(ii)] $P_i\cap P_j=\emptyset$ for $i\not=j$ and
\item[(iii)] $\bigcup_{i=1}^n P_i=\{x_1,x_2,\ldots,x_m\}$.
\end{enumerate}
Let $R_i=\la\la P_i\ra\ra$ be the normal closure of $P_i$ in $G$. Then
$$
R_1\cap R_2\cap\cdots\cap R_n=[[R_1,R_2],\ldots,R_n]_S.
$$\hfill $\Box$
\end{thm}

\subsection{Homotopy Groups of Higher Dimensional Spheres and Free Products of Surface Groups}
It is a natural question whether one can get similar group
theoretical descriptions of the (general) homotopy groups of
higher-dimensional spheres. We give some remarks that the homotopy
groups of certain $2$-dimensional complexes contains the homotopy
groups of all of higher-dimensional spheres as summands. From this,
we can answer the above question in some sense.

Let $X$ and $Y$ be path-connected spaces. According to~\cite{Gray}, there is a homotopy decomposition
\begin{equation}\label{equation4.2}
\Omega(X\vee Y)\simeq \Omega X\times \Omega Y\times \Omega\Sigma(\Omega X\wedge \Omega Y).
\end{equation}

Now let $X=S^2$ and let $Y$ be any surface. Equation~(\ref{equation4.2}) implies the following homotopy decomposition
\begin{equation}\label{equation4.3}
\Omega(S^2\vee Y)\simeq  \Omega S^2\times \Omega Y\times\Omega(\Sigma \Omega S^2\wedge \Omega Y).
\end{equation}
From the classical James Theorem~\cite{James1}, there is homotopy decomposition
$$
\Sigma\Omega\Sigma Z\simeq \bigvee_{k=1}^\infty \Sigma Z^{\wedge k},
$$
where $Z^{\wedge k}$ is the $k$-fold self smash product of $Z$. Note that the $k$-fold self smash product of $S^1$ is $S^k$. There is a homotopy decomposition
$$
\Sigma\Omega S^2=\Sigma\Omega\Sigma S^1\simeq \bigvee_{k=1}^\infty \Sigma S^k=\bigvee_{k=2}^\infty S^k.
$$
By substituting this decomposition into formula~(\ref{equation4.3}), there is a homotopy decomposition
$$
\begin{array}{rcl}
\Omega(S^2\vee Y)&\simeq & \Omega S^2\times \Omega Y\times\Omega \left(\bigvee_{k=2}^\infty S^k\wedge\Omega Y\right)\\
&\simeq& \Omega S^2\times \Omega Y \times \left(\prod_{k=2}^\infty \Omega (S^k\wedge \Omega Y) \right) \times\textrm{ other factors}.\\
\end{array}
$$
If $Y$ is a surface with $Y\not=S^2$, $\RP^2$ and $D^2$, then $\Omega Y$ is homotopy equivalent to a discrete space of a countably infinite set. Then $S^k\wedge Y$ is a wedge of countably infinite copies of $S^k$. Thus given any $k\geq2$,
$\pi_q(S^k)$ is a summand of $\pi_q(S^2\vee Y)$ for each $q$ with infinite multiplicity. If $Y=S^2$ or $\RP^2$, then we can repeat the above decomposition formula and obtain the fact that, given any $k\geq 2$, $\pi_q(S^k)$ is a summand of $\pi_q(S^2\vee Y)$ for each $q$ (with infinite multiplicity).

Hence it suffices to consider $\pi_*(S^2\vee Y)$ for any surface $Y\not=D^2$. By the same arguments in the previous subsection, we can get partitions of the space $S^2\vee Y$ by removing points from $S^2$ and $Y$. This gives the following result.

\begin{thm}
Let $G_1=\la x_{11},x_{12},\ldots,x_{1m} \ | \ x_{11}x_{12}\cdots x_{1m}=1\ra$ and let $G_2$ be one of the following surface groups:
\begin{enumerate}
\item[(a)] $\la a_1,b_1,\ldots,a_g,b_g, y_1,\ldots,y_t, x_{21},\ldots,x_{2m} \ | \ \prod\limits_{i=1}^g[a_i,b_i]=\prod\limits_{j=1}^ty_j\cdot \prod\limits_{i=1}^m x_{2i}\ra$.
\item[(b)]  $\la a_1,\ldots,a_h, y_1,\ldots,y_t, x_{21},\ldots,x_{2m} \ | \ \prod\limits_{i=1}^h a_i^2=\prod\limits_{j=1}^ty_j\cdot \prod\limits_{i=1}^m x_{2i}\ra$.
\end{enumerate}
Let $P_{ji}$ be any subset of $\{x_{j1},x_{j2},\ldots,x_{jm}\}$, $j=1,2$ and $1\leq i\leq n$, such that
\begin{enumerate}
\item[(i)] $P_{ji}\not=\emptyset$ for each $1\leq i\leq n$;
\item[(ii)] $P_{ji}\cap P_{jk}=\emptyset$ for $i\not=k$ and
\item[(iii)] $\bigcup_{i=1}^n P_{ji}=\{x_{j1},x_{j2},\ldots,x_{jm}\}$.
\end{enumerate}
Let $R_i=\la\la P_{1i},P_{2i}\ra \ra$ be the normal closure of $P_{1i}\cup P_{2i}$ in the free product $G=G_1\ast G_2$. Then there is an isomorphism of groups
$$
\pi_n(S^2\vee Y)\cong (R_1\cap R_2\cap\cdots\cap R_n)/[[R_1,R_2],\ldots,R_n]_S,
$$
where $Y$ is a surface such that the fundamental group of $Y$ punctured by $m$ points is the group $G_2$.\hfill $\Box$
\end{thm}

\section{Applications to Braid Groups and the Proof of Theorem~\ref{theorem1.5}}\label{section5}

\subsection{Applications to the Braid Groups}
Recall that a braid of $n$ strands is called \textit{Brunnian} if deleting any one of the strands produces a trivial braid of $(n-1)$-strands. Let $\Brun_n$ denote the Brunnian subgroup of the $n\,$th pure Artin braid group $P_n$. (\textbf{Note.} An $n$-strand Brunnian braid is always a pure braid for $n\geq3$ according to ~\cite[Proposition 4.2.2]{BCWW}.) The group $\Brun_n$ has been characterized by Levinson ~\cite{Levinson1,Levinson2} for $n\leq4$. A classical question proposed by Makanin~\cite{Makanin} in 1980 is to determine a set of generators for $\Brun_n$. This question has been answered by Johnson~\cite{Johnson} and Stanford~\cite{Stanford}. An answer using fat commutators was given in~\cite[Theorem 8.6.1]{BCWW}. As an application, we gave a generalization of Levinson's result~\cite[Theorem 2]{Levinson2}.

Recall that the braid group $B_n$ is generated by $\sigma_1,\ldots,\sigma_{n-1}$ with defining relations given by
\begin{enumerate}
\item $\sigma_i\sigma_j=\sigma_j\sigma_i$ for $|i-j|\geq 2$;
\item $\sigma_i\sigma_{i+1}\sigma_i=\sigma_{i+1}\sigma_i\sigma_{i+1}$.
\end{enumerate}
Following Levinson's notation in~\cite{Levinson2}, let
$$
t_i=\sigma_i\sigma_{i+1}\cdots\sigma_{n-2}\sigma_{n-1}^2\sigma_{n-2}^{-1}\cdots\sigma_i^{-1}
$$
for $1\leq i\leq n-1$. Intuitively $t_i$ is the braid that links strand $i$ and strand $n$ in front of all other strands. Let $R_i=\la\la t_i\ra\ra$ be the normal closure of $t_i$ in $P_n$. (\textbf{Note.} Levinson uses the notation $\theta_i$.)
\begin{thm}\label{theorem5.1}
For each $n\geq2$, $\Brun_n=[[R_1,R_2],\ldots,R_{n-1}]_S$.
\end{thm}
 When $n=4$, we have $\Brun_4=[[R_1,R_2],R_3]_S=[[R_1,R_2],R_3]\cdot [[R_1,R_3],R_2]$, which is exactly Levinson's theorem~\cite[Theorem 2]{Levinson2}. Let $M$ be any manifold. Recall that the coordinate projection
$$
F(M,n)\longrightarrow F(M,k)\quad (z_1,\ldots,z_n)\mapsto (z_{i_1},\ldots,z_{i_k}) \textrm{ for given } i_1<\cdots<i_k
$$
is a fibration by Fadell-Neuwirth's Theorem~\cite{FN}. The coordinate projection
$$
d_j\colon F(\R^2,n)\longrightarrow F(\R^2,n-1),\quad (z_1,\ldots,z_n)\mapsto (z_1,\ldots,\hat z_j,\ldots, z_n)$$ induces a group homomorphism
$$
d_{j*}\colon P_n=\pi_1(F(\R^2,n))\longrightarrow P_{n-1}=\pi_1(F(\R^2,n-1))
$$
which is given by removing the $j\,$th strand.
\begin{proof}[Proof of Theorem~\ref{theorem5.1}]
The basepoint $(q_1, q_2, \ldots, q_n)$ of $F(\R^2,n)$ is chosen in the way that the points $q_i\in\R^1\subseteq \R^2$ with the order $q_1<q_2<\cdots<q_n$. Let $Q_n=\{q_1,\ldots,q_n\}$ and let $Q_{n,i}=\{q_1,\ldots,\hat q_i,\ldots,q_n\}$. For each $1\leq i\leq n-1$, there is a commutative diagram of fibrations
\begin{diagram}
\R^2\smallsetminus Q_{n-1}&\rInto^{f}& F(\R^2,n)&\rOnto^{d_n}& F(\R^2,n-1)\\
\dInto>{g_i}&&\dTo>{d_i}&&\dTo>{d_i}\\
\R^2\smallsetminus Q_{n-1,i}&\rInto^{h_i}& F(\R^2,n-1)&\rOnto^{d_{n-1}}& F(\R^2,n-2).\\
\end{diagram}
By taking fundamental groups and using the fact that $F(\R^2,m)$ is a $K(\pi,1)$-space, there is a commutative diagram of short exact sequences of groups
\begin{diagram}
\pi_1(\R^2\smallsetminus Q_{n-1})&\rInto^{f_{*}}& P_n=\pi_1(F(\R^2,n))&\rOnto^{d_{n*}}& P_{n-1}=\pi_1(F(\R^2,n-1))\\
\dInto>{g_{i*}}&&\dTo>{d_{i*}}&&\dTo>{d_{i*}}\\
\pi_1(\R^2\smallsetminus Q_{n-1,i})&\rInto^{h_{i*}}& P_{n-1}=\pi_1(F(\R^2,n-1))&\rOnto^{(d_{n-1})_*}& P_{n-2}=\pi_1(F(\R^2,n-2)).\\
\end{diagram}
Let
$$
R'_i=\Ker(g_{i*}\colon \pi_1(\R^2\smallsetminus Q_{n-1})\to \pi_1(\R^2\smallsetminus Q_{n-1,i})).
$$
Since $h_{i*}$ is a monomorphism,
$$
f_{*}|\colon R'_i\longrightarrow \Ker(d_{n*}\colon P_n\to P_{n-1})\cap\Ker(d_{i*}\colon P_n\to P_{n-1})
$$
is an isomorphism. Observe that $R'_i$ is the normal closure of the elements in $\pi_1(\R^2\smallsetminus Q_{n-1})$ represented by the small circles around the point $q_i$. By using the terminology of braids, the small circles around the point $q_i$ is represented by ~$t_i^{\pm}$. Thus
$$
\Ker(d_{n*}\colon P_n\to P_{n-1})\cap\Ker(d_{i*}\colon P_n\to P_{n-1})\leq R_i.
$$
On the other hand, since $d_{n*}(t_i)=d_{i*}(t_i)=1$, we have
$$
R_i\leq \Ker(d_{n*}\colon P_n\to P_{n-1})\cap\Ker(d_{i*}\colon P_n\to P_{n-1})
$$
and so
\begin{equation}\label{R_i}
\Ker(d_{n*}\colon P_n\to P_{n-1})\cap\Ker(d_{i*}\colon P_n\to P_{n-1})=R_i
\end{equation}
with the property that
$$
f_*\colon R'_i\longrightarrow R_i
$$
is an isomorphism. If $w\in \bigcap\limits_{i=1}^{n-1}R_i$, then $w\in \mathrm{Im}(f_*)$ and $f_*^{-1}(w)\in R'_i$ for each $1\leq i\leq n-1$. Thus the monomorphism $f_*$ induces an isomorphism
$$
f_*\colon \bigcap_{i=1}^{n-1}R'_i\longrightarrow \bigcap_{i=1}^{n-1} R_i.
$$
By applying Theorem~\ref{theorem4.3} to the punctured disks, we have
$$
R'_1\cap R'_2\cap\cdots\cap R'_{n-1}=[[R'_1,R'_2],\ldots,R'_{n-1}]_S.
$$
It follows that
$$
R_1\cap R_2\cap\cdots\cap R_{n-1}=[[R_1,R_2],\ldots,R_{n-1}]_S.
$$
From equality~(\ref{R_i}), we have
$$
\bigcap_{i=1}^{n-1} R_i=\bigcap_{i=1}^{n-1}\Ker(d_{n*})\cap\Ker(d_{i*})=\bigcap_{j=1}^{n}\Ker(d_{j*}\colon P_n\to P_{n-1})=\Brun_n,
$$
hence the result.
\end{proof}

\subsection{Proof of Theorem~\ref{theorem1.5}}
The key point to proving Theorem~\ref{theorem1.5} is to construct a $K(\pi,1)$-partition for the configuration space $F(S^2,n)$.
Before giving the proof, we need some lemmas. Let $q_1,q_2,\ldots$ be strictly increasing real numbers that lie in the open interval $(-1,1)\subseteq \R^1\subseteq \R^2\subseteq S^2$. The base point of the configuration spaces $F(\R^2,m)\subseteq F(S^2,m)$ is chosen to be $(q_1,\ldots,q_m)$ for each $m\geq 1$.

\begin{lem}\label{lemma5.2}
Let $1\leq i\leq m$ and let
$$
B_i=\{(z_1,\ldots,z_m)\in F(S^2,m)\ | \ z_p\not=\infty\textrm{ for } p\not=i\}.
$$
Let $f\colon F(\R^2,m)\to B_i$ be the inclusion. Then the homomorphism
$$
f_*\colon P_m=\pi_1(F(\R^2,m))\longrightarrow \pi_1(B_i)
$$
is an epimorphism with $\Ker(f_*)=\la\la A_{0,i}\ra\ra^{P_m}$, the normal closure of $A_{0,i}$ in $P_m$.
\end{lem}
\begin{proof}
The assertion holds trivially for $m=1$. We may assume that $m\geq 2$. Consider the Fadell-Neuwirth fibration
$$
d_i\colon F(S^2,m)\longrightarrow F(S^2,m-1)\quad (z_1,\ldots,z_n)\mapsto (z_1,\ldots,\hat z_i,\ldots,z_m).
$$
Then $B_i=d_i^{-1}(F(\R^2,m-1))$ by the construction of $B_i$. Consider the commutative diagram of the fibrations
\begin{diagram}
\R^2\smallsetminus Q_{m,i}&\rInto& F(\R^2,m)&\rTo& F(\R^2,m-1)\\
\dInto>{f|}&&\dInto>{f}&&\dEq\\
S^2\smallsetminus Q_{m,i}&\rInto& B_i&\rTo&F(\R^2,m-1),\\
\end{diagram}
where $Q_{m,i}=\{q_1,\ldots,q_{i-1},q_{i+1},\ldots,q_m\}$. Thus there is a commutative diagram of short exact sequence of groups
\begin{diagram}
\pi_1(\R^2\smallsetminus Q_{m,i})&\rInto& \pi_1(F(\R^2,m))&\rOnto& \pi_1(F(\R^2,m-1))\\
\dTo>{f|_*}&&\dTo>{f_*}&&\dEq\\
\pi_1(S^2\smallsetminus Q_{m,i})&\rInto& \pi_1(B_i)&\rOnto&\pi_1(F(\R^2,m-1)).\\
\end{diagram}
Since $f|_*\colon \pi_1(\R^2\smallsetminus Q_{m,i})\to \pi_1(S^2\smallsetminus Q_{m,i})$ is an epimorphism, the middle map
$$
f_*\colon \pi_1(F(\R^2,m))\longrightarrow \pi_1(B_i)
$$
is an epimorphism with $\Ker(f_*)=\Ker(f|_*)$. Note that $\Ker(f|_*)$ is the normal closure of the element $[\omega_i]$ in $\pi_1(\R^2\smallsetminus Q_{m,i})$, where the loop $\omega_i$ is given by the following left picture:
\begin{center}
\epsfig{file=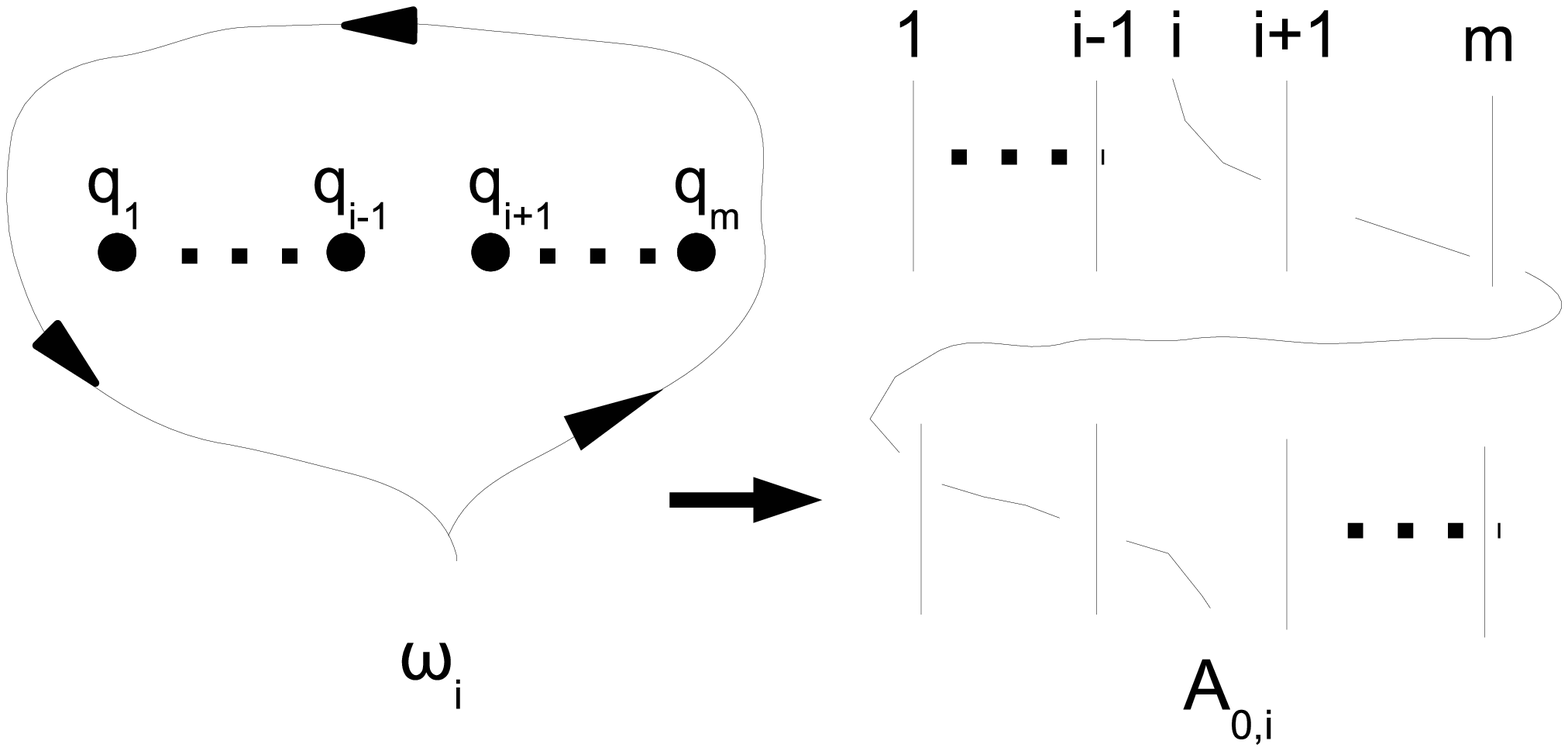, bbllx=0, bblly=0, bburx=488, bbury=72, width=4in,
height=4in, clip=}
\end{center}
\vspace*{-5cm}
In the above picture, the braid $A_{0,i}$ is obtained
from the loop $\omega_i$ by drawing it as a braid. Thus
$[\omega_i]=A_{0,i}$ and so
$$
\Ker(f_*)=\Ker(f|_*)=\la\la A_{0,i}\ra\ra^{P_m}\cap \pi_1(\R^2\smallsetminus Q_{m,i})\leq \la\la A_{0,i}\ra\ra^{P_m}.
$$
On the other hand, since $\Ker(f_*)$ is normal in $P_m$ with $A_{0,i}\in \Ker(f_*)$, we have $\la\la A_{0,i}\ra\ra^{P_m}\leq \Ker(f_*)$. It follows that
$$
\Ker(f_*)=\la\la A_{0,i}\ra\ra^{P_m},
$$
hence the result holds.
\end{proof}

\begin{lem}\label{lemma5.3}
Let $I=\{i_1,i_2,\ldots,i_k\}$ be a subset of $\{1,2,\ldots,m\}$. Let
$$
B_I=\{(z_1,\ldots,z_m)\in F(S^2,m) \ | \ z_p\not=\infty\textrm{ for } p\not\in I\}
$$
and let $f\colon F(\R^2,m)\to B_I$ be the inclusion. Then the homomorphism
$$
f_*\colon \pi_1(F(\R^2,m))\to \pi_1(B_I)
$$
is an epimorphism whose kernel
$
\Ker(f_*)=\la\la A_{0,i}\ | \ i\in I\ra\ra^{P_m}.
$
\end{lem}
\begin{proof}
Note that
$$
B_I=\bigcup_{s=1}^k B_{i_s}.
$$
By Lemma~\ref{lemma5.2}, $P_m=\pi_1(F(\R^2,m))\to \pi_1(B_{i_s})$ is an epimorphism whose kernel is given by $\la\la A_{0,i_s}\ra\ra^{P_m}$ for each $1\leq s\leq k$. From Van Kampen's Theorem,
$$
f_*\colon P_m=\pi_1(F(\R^2,m))\longrightarrow \pi_1(B_I)
$$
is an epimorphism with
$$
\Ker(f_*)=\prod_{s=1}^k \la\la A_{0,i_s}\ra\ra=\la\la A_{0,i_s}\ | \ 1\leq s\leq k\ra\ra^{P_m},
$$
hence the result.
\end{proof}

\begin{proof}[Proof of Theorem~\ref{theorem1.5}]
 Let $A=F(\R^2,m)$ be the subspace of $F(S^2,m)$. Let
$$
A_i=\{(z_1,z_2,\ldots,z_m)\in F(S^2,m)\ | \ z_p\not=\infty \textrm{ for } p\not\in \Lambda_i\},
$$
where the $\Lambda_i$ are as defined in the statement of the Theorem.
Since $\Lambda_i\cap \Lambda_j=\emptyset$ for $i\not=j$, we have
$$
A_i\cap A_j=F(\R^2,m)=A
$$
for $i\not=j$. From the assumption that $\bigcup_{i=1}^n \Lambda_i=\{1,2,\ldots,m\}$, we have
$$
\bigcup_{i=1}^n A_i=F(S^2,m).
$$
Thus $(A_1,A_2,\ldots,A_n)$ is a cofibrant $n$-partition \footnote{Similar to footnote$^2$, one needs to do some modifications such that the cofibrant hypothesis holds: Take a triangulation on $F(S^2,m)$ and deform linearly $A$ and each $A_i$ into certain subcomplexes with necessary subdivisions for having a cofibrant $n$-partition of $F(S^2,m)$.}  of the space $F(S^2,m)$ relative to $F(\R^2,m)$. For a subset $I=\{i_1,i_2,\ldots,i_k\}\subsetneq\{1,2,\ldots,n\}$, let
$$
A_I=\bigcup_{s=1}^k A_{i_s}
$$
with $A_{\emptyset}=A$. We now check that $A_I$ is a $K(\pi,1)$-space for $I\subsetneq \{1,2,\ldots,n\}$. Let
$$
J=\{j_1,j_2,\ldots,j_{t}\}=\{1,2,\ldots,m\}\smallsetminus \left(\bigcup_{s=1}^k\Lambda_{i_s}\right)
$$
with $j_1<j_2<\ldots<j_t$. Since $I\not=\{1,2,\ldots,n\}$, $J$ is not empty. Let $(z_1,\ldots,z_m)\in F(S^2,m)$. Observe that $(z_1,\ldots,z_m)\in A_I$ if and only if $z_{j_s}\not=\infty$ for $1\leq s\leq t$.  Consider the Fadell-Neuwirth fibration~\cite{FN}
$$
d_I\colon F(S^2,m)\longrightarrow F(S^2,t).
$$
There is a pull-back diagram
\begin{diagram}
A_I&\rInto& F(S^2,m)\\
d_I|_{A_I}\dTo&\mathrm{pull}&\dTo>{d_I}\\
F(\R^2,t)&\rInto&F(S^2,t)\\
\end{diagram}
with a fibration
$$
F(S^2\smallsetminus Q_t, m-t)\rTo A_I\rTo F(\R^2,t),
$$
where $Q_t$ is a set of $t$ distinct points in $S^2$. Since $t\geq 1$,
$$
F(S^2\smallsetminus Q_t,m-t)\cong F(\R^2\smallsetminus Q_{t-1}, m-t)
$$
is a $K(\pi,1)$-space. Together with the fact that the base space $F(\R^2,t)$ is a $K(\pi,1)$-space, the total space $A_I$ is a $K(\pi,1)$-space for any $I\subsetneq \{1,2,\ldots,n\}$.

By Lemma~\ref{lemma5.3}, the inclusion $f\colon A=F(\R^2,m)\hookrightarrow A_i$ induces an epimorphism
$$
f_*\colon P_m=\pi_1(A)\rOnto \pi_1(A_i)
$$
whose kernel
$$
R_i=\Ker(f_*)=\la\la A_{0,j}\ | \ j\in \Lambda_i\ra\ra^{P_m}.
$$
The assertion follows from Theorem~\ref{theorem1.3}.
\end{proof}

\noindent\small{\textbf{Acknowledgements.}} The authors would like to thank the referee for pointing out some errors in the original manuscript. The authors also would like to thank Fred Cohen for his help on the writing of this article.


\begin{thebibliography}{MMM}
\bibitem{BCWW}\auths{J. A. Berrick, F. R. Cohen, Y. L. Wong and J. Wu} \textsl{Braids, configurations and homotopy groups}, \textrm{J. Amer. Math. Soc.} \Vol{19}\Year{2006}, \Pages{265--326}.
\bibitem{BMVW}\auths{V. G. Bardakov, R. Mikhailov, V. V. Vershinin and J. Wu} \artTitle{Brunnian braids on surfaces} preprint.
\bibitem{BL} \auths{R. Brown and J.-L. Loday} \artTitle{Van Kampen theorems for diagrams of
spaces} \jTitle{Topology} \Vol{26} \Year{1987}, \Pages{311-335}.
\bibitem{CW1}
\auths{F. R. Cohen and J. Wu}
\artTitle{On braids, free groups and the loop space of the 2-sphere}
\jTitle{Progress in Mathematic Techniques}
\Vol{215}
\Year{2004}, Algebraic Topology: Categorical Decompositions, \Pages{93-105}.

\bibitem{CW2}
\auths{F. R. Cohen and J. Wu}
\artTitle{On braid groups and homotopy groups}
\jTitle{Geometry \& Topology Monographs} \Vol{13} \Year{2008}, \Pages{169-193}.
\bibitem{EM} \auths{G. Ellis and R. Mikhailov} \artTitle{A colimit of classifying spaces} arXiv:0804.3581v1 [Math. GR] 22 April 2008.
\bibitem{Fadell} \auths{E. Fadell} \artTitle{Homotopy groups of configuration spaces and the string problem of Dirac} \jTitle{Duke Math. J.} \Vol{29} \Year{1962}, \Pages{231-242}.
\bibitem{FB} \auths{E. Fadell and J. Van Buskirk} \artTitle{The braid groups of $E^{2}$ and $S^{2}$} \jTitle{Duke Math. J.} \Vol{29} \Year{1962}, \Pages{243-257}.
\bibitem{FN} \auths{E. Fadell and L. Neuwirth} \artTitle{Configuration spaces}  \jTitle{Math. Scand} \Vol{10} \Year{1962}, \Pages{111-118}.
\bibitem{Gray} \auths{B. Gray} \artTitle{A note on the Hilton-Milnor theorem} \jTitle{Topology} \Vol{10} \Year{1971}, \Pages{199--201}.
\bibitem{James1} \auths{I. M. James} \artTitle{Reduced product spaces} \jTitle{Ann. Math.} \Vol{62} \Year{1953}, \Pages{170-197}.
\bibitem{Johnson} \textrm{D. L. Johnson}, \textsl{Towards a characterization
of smooth braids}, \textrm{Math. Proc. Cambridge Philos. Soc.}
\textbf{92} \textrm{(1982)}, \textrm{425--427}.
\bibitem{Levinson1}\auths{H. W. Levinson} \artTitle{Decomposable braids and linkages} \jTitle{Trans. AMS} \Vol{178}\Year{1973}, \Pages{111-126}.
\bibitem{Levinson2}\auths{H. W. Levinson} \artTitle{Decomposable braids as subgroups of braid groups} \jTitle{Trans. AMS} \Vol{202}\Year{1975}, \Pages{51-55}.
\bibitem{LW} \textrm{ J. Y. Li and J. Wu}, \textsl{Artin braid groups and homotopy groups}, \textrm{Proc. London Math. Soc.} \textbf{99} \textrm{(2009)},
\textrm{521--556}.
\bibitem{MKS}
\auths{W. Magnus, A. Karrass and D. Solitar}
\artTitle{Combinatorial group theory}
\jTitle{Pure and Applied Mathematics}
\Vol{XIII}\Year{1966}.
\bibitem{Makanin}\auths{G. S. Makanin} \artTitle{Kourovka Notebook (Unsolved Questions  in Group Theory) Seventh edition. Novosibirsk. 1980} \text{Question 6.23.} \textrm{p.~78}.
\bibitem{Stanford}\auths{T. Stanford} \artTitle{Brunnian braids and some of their generalizations} arXiv: math/9907072v1
\bibitem{Wu1} \textrm{J. Wu}, \textsl{Combinatorial descriptions of the
homotopy groups of certain spaces}, \textrm{Math. Proc. Camb.
Philos. Soc.} \textbf{130} \textrm{(2001)}, \textrm{489--513}.
\bibitem {Wu2}\textrm{J. Wu}, \textsl{A braided simplicial group},
\textrm{Proc. London Math. Soc.} \textbf{84} \textrm{(2002)},
\textrm{645--662}.
\end{thebibliography}
\end{document}